\documentclass[reqno]{amsart}
\usepackage{latexsym,amssymb}             
\usepackage{amsmath}                      
\usepackage{amsthm}                       
\usepackage{amscd}                        
\usepackage{mathtools, bbm,enumerate,xcolor,hyperref}

\makeatletter
\@addtoreset{figure}{section}
\def\thefigure{\thesection.\@arabic\c@figure}
\def\fps@figure{h, t}
\@addtoreset{equation}{section}

\makeatother

        \theoremstyle{plain} 
        \newtheorem{theorem}             {Theorem}  [section]
        \newtheorem{lemma}      [theorem]{Lemma}
        \newtheorem{corollary}  [theorem]{Corollary}
        \newtheorem{proposition}[theorem]{Proposition}

		\newtheorem{thmalph}			{Theorem}
		
		\newtheorem{coralph}	[thmalph]{Corollary}

        \theoremstyle{definition}

        \newtheorem{example}    [theorem]{Example}

        \theoremstyle{remark}
        \newtheorem{remark}     [theorem]{Remark}

\begin{document}

\title[Explicit Hilbert spaces for the unitary dual of rank one orthogonal groups]{Explicit Hilbert spaces for the unitary dual of rank one orthogonal groups and applications}

\author{Christian Arends, Frederik Bang-Jensen and Jan Frahm}
\address{Department of Mathematics, Aarhus University, Ny Munkegade 118, 8000 Aarhus C, Denmark}
\email{arends@math.au.dk, frederik\_bang-jensen@hotmail.com, frahm@math.au.dk}

\begin{abstract}
We realize all irreducible unitary representations of the group $\mathrm{SO}_0(n+1,1)$ on explicit Hilbert spaces of vector-valued $L^2$-functions on $\mathbb{R}^n\setminus\{0\}$. The key ingredient in our construction is an explicit expression for the standard Knapp--Stein intertwining operators between arbitrary principal series representations in terms of the Euclidean Fourier transform on a maximal unipotent subgroup isomorphic to $\mathbb{R}^n$.\\
As an application, we describe the space of Whittaker vectors on all irreducible Casselman--Wallach representations. Moreover, the new realizations of the irreducible unitary representations immediately reveal their decomposition into irreducible representations of a parabolic subgroup, thus providing a simple proof of a recent result of Liu--Oshima--Yu.
\end{abstract}

\maketitle

\section*{Introduction}

The unitary dual $\widehat{G}$ of a Lie group $G$ is the set of equivalence classes of its irreducible unitary representations. While the unitary dual of general Lie groups, in particular of most semisimple groups, is still not completely known, there are some subclasses of groups where we have a full description. The topic of this paper is the unitary dual of the indefinite orthogonal group $G=\mathrm{SO}_0(n+1,1)$, $n>1$, which has been known since the 1960s by the work of Hirai~\cite{Hirai62} (see also \cite{BaldoniSilvaBarbasch83,BransonOlafsson,Thieleker73} for alternative approaches).

Although a unitary representation of $G$ is a representation on a (typically infinite-dimensional) Hilbert space, many of the irreducible unitary representations are constructed by completing a representation of the Lie algebra $\mathfrak{g}$ of $G$ with respect to an algebraically defined inner product. This makes the Hilbert space on which the group acts a rather abstract object. However, when studying analytic problems in representation theory such as branching problems where restricted representations are decomposed into direct integrals of Hilbert spaces, it is of utmost importance to have a more intrinsic understanding of the Hilbert spaces on which the representations are realized (see e.g. \cite{MoellersOshima15}).

The goal of this paper is to provide such realizations for all irreducible unitary representations of $G$. Every such representation can be realized as a subrepresentation of a principal series representation induced from a finite-dimensional representation of a parabolic subgroup $P$ of $G$. Those principal series representations which are induced from unitary representations of $P$ are by construction unitary with respect to a natural $L^2$-inner product, so the corresponding Hilbert space is a space of (possibly vector-valued) $L^2$-functions. But there exist principal series representations induced from non-unitary representations of $P$ or even subrepresentations thereof which carry a more involved invariant inner product given in terms of the standard Knapp--Stein intertwining operators. What makes the corresponding Hilbert spaces more delicate is the fact that the inner product is given in terms of convolution with a possibly singular operator-valued integral kernel. A priori it is not even clear whether the resulting Hermitian forms are positive (semi)definite.

The key idea is to turn the complicated convolution type Hermitian forms into more accessible ones by applying a Fourier transform. More precisely, the principal series representations can be realized in the so-called non-compact picture on a space of (possibly vector-valued) functions on a maximal unipotent subgroup $\overline{N}\simeq\mathbb{R}^n$. In these coordinates, the Knapp--Stein operators are convolution operators, so the Euclidean Fourier transform turns them into operator-valued multiplication operators. The main result of this paper is an explicit expression for these multiplication operators in terms of their eigenspaces and eigenvalues. The unitarity question then translates to the simple question of positivity of these eigenvalues, and the corresponding Hilbert spaces are $L^2$-spaces of functions with values in the corresponding eigenspaces.

The idea of utilizing the Fourier transform (or the related Laplace transform) for such problems is certainly not new, see e.g. \cite{Ding99,DingGross93,DvorskySahi99,DvorskySahi03,HilgertKobayashiMoellers14,KobayashiOersted03,MoellersOshima15,MoellersSchwarz17,Sahi92,VergneRossi76}. However, so far it has only been applied to principal series representations induced from one-dimensional representations. Since most of the representations in the unitary dual of $G$ are contained in principal series representations induced from higher-dimensional representations of $P$, we have to deal with vector-valued functions and operator-valued integral kernels, making the analysis much more subtle.

We now give a more detailed description of our results.

\subsection*{Principal series representations and intertwining operators}

Let $G=\mathrm{SO}_0(n+1,1)$, the identity component of the group of real matrices of size $n+2$ preserving the standard quadratic form of signature $(n+1,1)$. The stabilizer $P$ in $G$ of the isotropic line through $(0,\ldots,0,1,1)\in\mathbb{R}^{n+2}$ is a parabolic subgroup and has a decomposition into $P=MAN$ with $M\simeq\mathrm{SO}(n)$, $A\simeq\mathbb{R}_+$ and $N\simeq\mathbb{R}^n$. For every irreducible representation $(\sigma,V_\sigma)$ of $M$ and for every (not necessarily unitary) character $e^\lambda$ parameterized by $\lambda\in\mathbb{C}$ we form the induced representation (smooth normalized parabolic induction)
$$ \pi_{\sigma,\lambda}=\mathrm{Ind}_P^G(\sigma\otimes e^\lambda\otimes1) $$
of $G$. The quotient $G/P$ contains an open dense subset isomorphic to a unipotent subgroup $\overline{N}\simeq\mathbb{R}^n$ opposite to $N$, so by restriction we can realize $\pi_{\sigma,\lambda}$ on a space $I_{\sigma,\lambda}\subseteq C^\infty(\mathbb{R}^n,V_\sigma)$ of $V_\sigma$-valued smooth functions on $\mathbb{R}^n$, the so-called \emph{non-compact picture}. An alternative way of viewing these representations is via conformal geometry: the group $G$ acts on $\mathbb{R}^n$ via conformal rational linear transformations and $(\pi_{\sigma,\lambda},I_{\sigma,\lambda})$ is the corresponding family of multiplier representations on $V_\sigma$-valued functions on $\mathbb{R}^n$.

For $\lambda\in i\mathbb{R}$, the character $e^\lambda$ is unitary and $\pi_{\sigma,\lambda}$ extends to a unitary representation of $G$ on $L^2(\mathbb{R}^n,V_\sigma)$ which is always irreducible, forming the so-called \emph{unitary principal series}. While this Hilbert space is rather explicit, the Hilbert space completion of other representations $\pi_{\sigma,\lambda}$ and their subrepresentations/quotients is more involved. First, for every self-dual representation $\sigma$ of $M$, there is an explicit constant $a_\sigma>0$ such that $\pi_{\sigma,\lambda}$ (with $\lambda\in\mathbb{R}$) is irreducible and unitarizable if and only if $\lambda\in(-a_\sigma,a_\sigma)$. These representations form the \emph{complementary series} and their invariant inner product can be written as
\begin{equation}
	I_{\sigma,\lambda}\times I_{\sigma,\lambda}\to\mathbb{C}, \quad (f_1,f_2)\mapsto\int_{\mathbb{R}^n}\int_{\mathbb{R}^n}\langle f_1(x),K_{\sigma,\lambda}(x-y)f_2(y)\rangle_\sigma\,dx\,dy,\label{eq:IntroInvInnerProd}
\end{equation}
where
$$ K_{\sigma,\lambda}(x) = \mathrm{const}\times|x|^{2\lambda-n}\sigma\left(I_n-2\frac{xx^t}{\|x\|^2}\right)\in\mathrm{End}(V_\sigma) \qquad (x\in\mathbb{R}^n\setminus\{0\}). $$
Here we extend the self-dual irreducible representation $\sigma$ of $M\simeq\mathrm{SO}(n)$ to $\mathrm{O}(n)$ and note that $I_n-2xx^t/\|x\|^2\in\mathrm{O}(n)$ for every $x\in\mathbb{R}^n\setminus\{0\}$. (Such an extension is unique up to twisting by the determinant character, so the above inner products only differ by a sign.) Moreover, the constant in $K_{\sigma,\lambda}(x)$ can be chosen such that $K_{\sigma,\lambda}$ forms a family of distributions in $\mathcal{S}'(\mathbb{R}^n)\otimes\mathrm{End}(V_\sigma)$ which depends holomorphically on $\lambda\in\mathbb{C}$ and is nowhere vanishing. We remark that it is a priori not clear that $K_{\sigma,\lambda}$ is positive definite for $\lambda\in(-a_\sigma,a_\sigma)$, and there is no obvious explicit description of the Hilbert space completion of $I_{\sigma,\lambda}\subseteq C^\infty(\mathbb{R}^n)\otimes V_\sigma$ with respect to the corresponding inner product (unless $\lambda=0$ where $K_{\sigma,\lambda}(x)=\mathrm{const}\times\delta(x)\cdot\mathrm{id}_{V_\sigma}$).

The kernels $K_{\sigma,\lambda}\in\mathcal{S}'(\mathbb{R}^n)\otimes\mathrm{End}(V_\sigma)$ arise from a family of intertwining operators $A_{\sigma,\lambda}:I_{\sigma,\lambda}\to I_{\sigma,-\lambda}$ (again $\sigma$ is assumed to be self-dual) in the sense that
$$ A_{\sigma,\lambda}f(x) = \int_{\mathbb{R}^n} K_{\sigma,\lambda}(x-y)f(y)\,dy \qquad (f\in I_{\sigma,\lambda},x\in\mathbb{R}^n), $$
the integral meant in the distribution sense. All other irreducible unitary representations of $G$ can be obtained as (one of at most two direct summands of) the kernel or image of $A_{\sigma,\lambda}$ and the invariant inner product is given by a similar expression as above, possibly regularized in $\lambda$.

\subsection*{The Fourier transform of intertwining operators}

The main result of this paper is a more accessible formula for the kernels $K_{\sigma,\lambda}\in\mathcal{S}'(\mathbb{R}^n)\otimes\mathrm{End}(V_\sigma)$ and hence for the intertwining operators $A_{\sigma,\lambda}$ in terms of the Euclidean Fourier transform. To state this formula, we first observe that the stabilizer $M_\xi$ of $\xi\in\mathbb{R}^n\setminus\{0\}$ in $M$ is isomorphic to $\mathrm{SO}(n-1)$. By the classical branching rules for the pair of groups $(\mathrm{SO}(n),\mathrm{SO}(n-1))$, the restriction of $\sigma$ to $M_\xi$ decomposes into a multiplicity-free finite direct sum of irreducible representations $\tau$ of $M_\xi$. We write
$$ V_\sigma = \bigoplus_{\tau\in\widehat{M}_\xi} W_\tau(\xi) \qquad (\xi\in\mathbb{R}^n\setminus\{0\}) $$
for the corresponding decomposition of the representation space $V_\sigma$, indicating its dependence of $\xi$. Note that since any two stabilizers $M_\xi$ are conjugate, we can use the same parametrization of the representations $\tau$ for all $\xi$.

\begin{thmalph}[see Theorem~\ref{thm:IntertwiningOperatorsFpicture}]\label{thm:Intro}
Let $\sigma\in\widehat{M}$ be self-dual. The Euclidean Fourier transform $\widehat{K}_{\sigma,\lambda}$ of $K_{\sigma,\lambda}\in\mathcal{S}'(\mathbb{R}^n)\otimes\mathrm{End}(V_\sigma)$ is smooth on $\mathbb{R}^n\setminus\{0\}$ and given by
\begin{equation}
	\widehat{K}_{\sigma,\lambda}(\xi) = \|\xi\|^{-2\lambda}\sum_\tau a_{\sigma,\lambda}(\tau)\cdot\mathrm{pr}_{W_\tau(\xi)} \in \mathrm{End}(V_\sigma) \qquad (\xi\in\mathbb{R}^n\setminus\{0\}),\label{eq:IntroFTKernel}
\end{equation}
with scalars $a_{\sigma,\lambda}(\tau)$ given explicitly in \eqref{eq:ExplicitIntertwiningScalars} in terms of the highest weights of $\sigma$ and $\tau$.
\end{thmalph}

In the special case $V_\sigma=\bigwedge^p\mathbb{C}^n$, this formula was previously obtained in \cite{FischmannOrsted2021} (see also Remark~\ref{rem:SpecialCasePforms}).

Applying the Plancherel Formula for the Euclidean Fourier transform as well as the fact that the Fourier transform turns a convolution into a multiplication, \eqref{eq:IntroFTKernel} allows us to rewrite the invariant inner product \eqref{eq:IntroInvInnerProd} as
\begin{equation}
	(f_1,f_2)\mapsto\sum_\tau a_{\sigma,\lambda}(\tau)\int_{\mathbb{R}^n\setminus\{0\}}\langle\mathrm{pr}_{W_\tau(\xi)}\widehat{f}_1(\xi),\mathrm{pr}_{W_\tau(\xi)}\widehat{f}_2(\xi)\rangle_\sigma\|\xi\|^{-2\lambda}\,d\xi.\label{eq:IntroInnerProductFpicture}
\end{equation}
The scalars $a_{\sigma,\lambda}(\tau)$ turn out to be positive for all $\tau$ if and only if $\lambda\in(-a_\sigma,a_\sigma)$, so we obtain a new proof for the existence of the complementary series together with an explicit description of the Hilbert space completion as the Fourier transform of a weighted vector-valued $L^2$-space:

\begin{coralph}[see Theorem~\ref{thm:InnerProductsInFpicture}~\eqref{thm:InnerProductsInFpicture2}]
Let $\sigma\in\widehat{M}$ be self-dual. The Hilbert space completion of the complementary series representation $I_{\sigma,\lambda}$, $\lambda\in(-a_\sigma,a_\sigma)$, is
\begin{equation}
	\left\{f\in\mathcal{S}'(\mathbb{R}^n)\otimes V_\sigma:\sum_\tau a_{\sigma,\lambda}(\tau)\int_{\mathbb{R}^n\setminus\{0\}}\|\mathrm{pr}_{W_\tau(\xi)}\widehat{f}(\xi)\|_\sigma^2\|\xi\|^{-2\lambda}\,d\xi<\infty\right\}.
\end{equation}
\end{coralph}

\subsection*{Irreducible subrepresentations and quotients}

By Theorem~\ref{thm:Intro}, we obtain the following explicit expression for the Fourier transform of the intertwining operator $A_{\sigma,\lambda}:I_{\sigma,\lambda}\to I_{\sigma,-\lambda}$:
\begin{equation}
	\widehat{A_{\sigma,\lambda}f}(\xi) = \|\xi\|^{-2\lambda}\sum_\tau a_{\sigma,\lambda}(\tau)\cdot\mathrm{pr}_{W_\tau(\xi)}\widehat{f}(\xi) \qquad (f\in I_{\sigma,\lambda},\xi\in\mathbb{R}^n\setminus\{0\}).\label{eq:IntroFourierTransformOfIntertwiner}
\end{equation}
All other irreducible unitary representations of $G$ arise as (one of at most two direct summands of) the kernel of $A_{\sigma,\lambda}$ for self-dual $\sigma\in\widehat{M}$ and $\lambda\in\mathbb{R}$ (see Proposition~\ref{prop:SubrepsAndQuotientsInFpicture} for the precise statement). The kernel of $A_{\sigma,\lambda}$ is
$$ \ker A_{\sigma,\lambda} = \bigoplus_{\tau:a_{\sigma,\lambda}(\tau)=0}I_{\sigma,\lambda}(\tau), \qquad \mbox{with }I_{\sigma,\lambda}(\tau) = \{f\in I_{\sigma,\lambda}:\widehat{f}(\xi)\in W_\tau(\xi)\mbox{ for all }\xi\in\mathbb{R}^n\setminus\{0\}\}. $$
The invariant inner product on the unitarizable subrepresentations is given by a similar formula as in \eqref{eq:IntroInnerProductFpicture} and the corresponding Hilbert space completion is again the Fourier transform of a weighted vector-valued $L^2$-space:

\begin{coralph}[see Theorem~\ref{thm:InnerProductsInFpicture}~\eqref{thm:InnerProductsInFpicture3} and \eqref{thm:InnerProductsInFpicture4}]
	If the subrepresentation $\ker A_{\sigma,\lambda}=\bigoplus_{\tau:a_{\sigma,\lambda}(\tau)=0}I_{\sigma,\lambda}(\tau)$ of $I_{\sigma,\lambda}$ is unitarizable, then the invariant inner product is given by
	$$ (f_1,f_2)\mapsto\sum_{\tau:a_{\sigma,\lambda}(\tau)=0}\frac{1}{a_{\sigma,-\lambda}(\tau)}\int_{\mathbb{R}^n\setminus\{0\}}\langle\mathrm{pr}_{W_\tau(\xi)}\widehat{f}_1(\xi),\mathrm{pr}_{W_\tau(\xi)}\widehat{f}_2(\xi)\rangle_\sigma\|\xi\|^{-2\lambda}\,d\xi $$
	and the corresponding Hilbert space completion equals
\begin{equation}
	\left\{f\in\mathcal{S}'(\mathbb{R}^n)\otimes V_\sigma:\begin{array}{cc}\widehat{f}(\xi)\in\bigoplus_{\tau:a_{\sigma,\lambda}(\tau)=0}W_\tau(\xi)\mbox{ for all }\xi\in\mathbb{R}^n\setminus\{0\}\mbox{ and}\\\sum_{\tau:a_{\sigma,\lambda}(\tau)=0}a_{\sigma,-\lambda}(\tau)^{-1}\int_{\mathbb{R}^n\setminus\{0\}}\|\mathrm{pr}_{W_\tau(\xi)}\widehat{f}(\xi)\|_\sigma^2\|\xi\|^{-2\lambda}\,d\xi<\infty\end{array}\right\}.
\end{equation}
\end{coralph}

We remark that our results can easily be translated to the disconnected groups $\mathrm{SO}(n+1,1)$ and $\mathrm{O}(n+1,1)$. However, we chose to work with the connected group $G=\mathrm{SO}_0(n+1,1)$ since here both the maximal compact subgroup $K$, the Levi factor $M$ and the stabilizers $M_\xi$ are connected and hence their irreducible representations are parametrized in terms of their highest weights.

\subsection*{Applications to branching laws and Whittaker vectors}

Applying the Euclidean Fourier transform to $I_{\sigma,\lambda}$ and its unitarizable subrepresentations yields new realizations for all irreducible unitary representations of $G$ on weighted vector-valued $L^2$-spaces. In these realizations, the action of the parabolic subgroup $\overline{P}$ opposite to $P$ is particularly simple. This allows us to immediately deduce the decomposition of the restriction of these unitary representations to $\overline{P}$ into irreducibles. In this way, we obtain a new and simple proof for the branching laws obtained by Liu--Oshima--Yu~\cite{LiuOshimaYu23} (see Corollary~\ref{cor:BranchingPbar}).

Moreover, restricting even further to the unipotent subgroup $\overline{N}$, we can use these new models to describe the space of Whittaker vectors on all Casselman--Wallach representations of $G$. In fact, by the theory of Jacquet integrals (see e.g. \cite[Theorem 15.4.1]{W92}), it is known that the space of Whittaker vectors on $I_{\sigma,\lambda}$ is naturally isomorphic to $V_\sigma'$. What is more difficult to describe is which Whittaker vectors factor through a quotient of $I_{\sigma,\lambda}$ resp. restrict to non-zero Whittaker vectors on a subrepresentation. Corollary~\ref{cor:WhittakerVectors} gives a precise answer to this question for all irreducible subrepresentations and quotients of $I_{\sigma,\lambda}$ and also identifies the structure of the space of Whittaker vectors as a representation of the stabilizer in $M$ of the respective character of $\overline{N}$. These results hold withouth the assumption on $\sigma$ to be self-dual, so they include indeed all irreducible admissible representations by the Casselman Embedding Theorem.

\subsection*{Idea of the proof}

Instead of trying to compute the Fourier transform $\widehat{K}_{\sigma,\lambda}(\xi)$ directly (which seems to be difficult in general), we make use of the intertwining properties of $A_{\sigma,\lambda}$. More precisely, using the fact that $A_{\sigma,\lambda}$ is intertwining for the action of $\overline{P}$ we first conclude that $\widehat{K}_{\sigma,\lambda}(\xi)$ has to be of the form in \eqref{eq:IntroFTKernel} for some scalars $a_{\sigma,\lambda}(\tau)$. Then we study in detail the Fourier transform of the Lie algebra action of $\mathfrak{n}$ to show a recurrence relation for the scalars $a_{\sigma,\lambda}(\tau)$, thus determining them up to a holomorphic and nowhere vanishing function in $\lambda$. This last step is inspired by the technique developed in \cite{BransonOlafsson}.

\subsection*{Acknowledgements}

The first named author was supported by a research grant from the Aarhus University Research Foundation (Grant No. AUFF-E-2022-9-34). The second and third named authors were supported by a research grant from the Villum Foundation (Grant No. 00025373). The third named author would like to thank Santosh Nadimpalli for inspiring discussions about Whittaker vectors on subrepresentations and quotients.

\section{The unitary dual of rank one orthogonal groups}

We recall from the literature the classification of all irreducible unitary representations of the group $G=\mathrm{SO}_0(n+1,1)$ in terms of the irreducible admissible representations and intertwining operators between them. We assume throughout the whole paper that $n>1$ since for $n=1$ the group $G$ is isomorphic to $\operatorname{PSL}(2,\mathbb{R})$ whose representation theory is well understood.

\subsection{Groups and decompositions}

Let $G=\mathrm{SO}_0 (n+1,1)$ denote the identity component of the group of all matrices in $\mathrm{GL} (n+2,\mathbb{R})$ that leave the quadratic form
\begin{align*}
    \mathbb{R}^{n+2}\to \mathbb{R}^{n+2},\quad x \mapsto x^{t}I_{n+1,1}x,\quad I_{n+1,1}=\mathrm{diag}(1,\ldots, 1,-1),
  \end{align*}
  invariant. We fix the Cartan involution $\theta (g)=(g^{t})^{-1}$ on $G$ and denote the corresponding involution on $\mathfrak{g}=\mathrm{Lie}(G)$ by the same letter. Then $\mathfrak{g}$ decomposes into the $+1$ and $-1$ eigenspaces $\mathfrak{k}$ and $\mathfrak{p}$ of $\theta $ on $\mathfrak{g}$:
  \begin{align*}
  \mathfrak{g}=\left\{
  \begin{pmatrix}
A & b \\
b^t & 0\\
\end{pmatrix}\mid A \in \mathfrak{so}(n+1),\ b \in \mathbb{R}^{n+1}\right\}=\left\{
  \begin{pmatrix}
A & 0 \\
0 & 0 \\
\end{pmatrix} \in \mathfrak{g} \right\} \oplus \left\{
  \begin{pmatrix}
0 & b \\
b^t  & 0 \\
\end{pmatrix}\in \mathfrak{g} \right\}=\mathfrak{k}\oplus \mathfrak{p}.
\end{align*}
We choose the maximal abelian subalgebra $\mathfrak{a} \coloneqq \mathbb{R} H_{0}$ of $\mathfrak{p}$, where $H_{0} \coloneqq E_{n+1,n+2}+E_{n+2,n+1}$ with $E_{i,j}$ denoting the $(n+2) \times (n+2)$-matrix whose $(i,j)$-entry is equal to $1$ and which is zero everywhere else. The root system for $(\mathfrak{g}, \mathfrak{a})$ consists of the roots $\pm \gamma $, with $\gamma \in \mathfrak{a}_{\mathbb{C} }^{*}$ such that $\gamma (H_0)=1$. We denote the corresponding root spaces by
\begin{align*}
\mathfrak{n}\coloneqq \mathfrak{g}_{\gamma },\quad \overline{\mathfrak{n}}\coloneqq \mathfrak{g}_{-\gamma }=\theta \mathfrak{n}
\end{align*}
and write $N \coloneqq \exp (\mathfrak{n}), \quad \overline{N}\coloneqq \exp (\overline{\mathfrak{n}})=\theta N$. Choosing $+\gamma$ as the positive root, the half sum of positive roots with multiplicities is given by $\rho =\frac{n}{2}\gamma $ and $\mathfrak{n}=\mathrm{span}\{N_1, \ldots , N_{n}\}$, $\overline{\mathfrak{n}}=\mathrm{span}\{\overline{N}_1, \ldots , \overline{N}_n\}$ with
\begin{align*}
N_j \coloneqq E_{j,n+1}-E_{j,n+2}-E_{n+1, j}-E_{n+2,j}, \qquad \overline{N}_j\coloneqq\theta(N_j)= E_{j,n+1}+E_{j,n+2}-E_{n+1, j}+E_{n+2,j}.
\end{align*}
For $x \in \mathbb{R}^{n}$ we set
\begin{align*}
n_x \coloneqq \exp \bigg(\sum_{j=1}^{n}x_j N_j\bigg), \quad \overline{n}_x \coloneqq \exp \bigg(\sum_{j=1}^{n}x_j \overline{N}_j\bigg).
\end{align*}
Moreover, let $K \coloneqq G^\theta\cong \mathrm{SO}(n+1)$ be the maximal compact subgroup of $G$ consisting of the fixed points of $\theta$ and put
\begin{align*}
A &\coloneqq \exp \mathfrak{a} = \left\{
  \begin{pmatrix}
I_n  & 0 & 0 \\
0 & \cosh (t) & \sinh(t) \\
0 & \sinh(t) & \cosh(t) \\
  \end{pmatrix}\mid t \in \mathbb{R} \right\}\\
M &\coloneqq Z_K (\mathfrak{a}) = \left\{
\begin{pmatrix}
A & 0 & 0 \\
0 & 1 & 0 \\
0 & 0 & 1 \\
\end{pmatrix}\mid A \in \mathrm{SO} (n) \right\}\cong \mathrm{SO}(n).
\end{align*}

Writing $P$ for the minimal parabolic subgroup $MAN$ we obtain the open dense Bruhat cell $\overline{N}P \subset G$. For generic $x,y\in\mathbb{R}^n$ the element $n_y^{-1}\overline{n}_x$ is again contained in $\overline{N}P$. In order to compute its decomposition as $\overline{n}_zme^{tH_0}n_w$, we use the representative $w_0 \coloneqq \mathrm{diag}(-1,\ldots ,-1,1)$ of the non-trivial Weyl group element in the slightly bigger group $\mathrm{O}(n+1,1)$.

\begin{lemma}\label{lemma:cnfnljqwca}
  For $x \in \mathbb{R}^n \setminus\{0\}$ we have $w_0 \overline{n}_x =\overline{n}_y me^{tH_0 }n_y $ with
  \begin{align*}
    y=-\frac{x}{\lVert x \rVert^{2}},\quad m=2 \frac{x x^t }{\lVert x \rVert^2}-I_n \in \mathrm{O}(n) \cong \mathrm{diag}(\mathrm{O}(n),1,1),\quad t=2 \log \lVert x \rVert.
  \end{align*}
\end{lemma}
\begin{proof}
  Note that
  \begin{align*}
    n_x =\begin{pmatrix}
      I_n & x & -x \\
      -x^t  & 1-\frac{\lVert x \rVert^2 }{2} & \frac{\lVert x \rVert^2 }{2} \\
      -x^t  & -\frac{\lVert x \rVert^2 }{2} & 1+\frac{\lVert x \rVert^2 }{2} \\
    \end{pmatrix} \quad \mbox{and} \quad \overline{n}_x =
    \begin{pmatrix}
      I_n  & x & x \\
      -x^t  & 1-\frac{\lVert x \rVert^2 }{2} & -\frac{\lVert x \rVert^2 }{2} \\
      x^t  & \frac{\lVert x \rVert^2 }{2} & 1+\frac{\lVert x \rVert^2 }{2} \\
    \end{pmatrix}.
  \end{align*}
  The result now follows by straightforward calculations.
\end{proof}

\begin{lemma}\label{lemma:cnfn1kagf2}
  For $x,y \in \mathbb{R}^n $ with $\lVert x \rVert^2 y+x \neq 0$ we have $n_y^{-1} \overline{n}_x =\overline{n}_z me^{tH_0 }n_1 \in \overline{N}P$ with
  \begin{align*}
    z=\frac{\lVert x \rVert^2 (x+\lVert x \rVert^2 y)}{\lVert x+\lVert x \rVert^2 y \rVert^2 },\ m=\left( 2 \frac{(x+\lVert x \rVert^2 y)(x+\lVert x \rVert^2 y)^t }{\lVert x+ \lVert x \rVert^2 y \rVert^2 }-I_n  \right)\left( 2 \frac{x x^t }{\lVert x \rVert^2 }- I_n  \right),\ t=2\log \frac{\lVert x+ \lVert x \rVert^2 y \rVert}{\lVert x \rVert}
  \end{align*}
  and some $n_1 \in N$.
\end{lemma}

\begin{proof}
  We first write $n_y^{-1} \overline{n}_x = w_0 (w_0 n_{-y}w_0 )w_0 \overline{n}_x=w_0 \overline{n}_{-y}w_0 \overline{n}_x$. By Lemma \ref{lemma:cnfnljqwca} we write $w_0 \overline{n}_x = \overline{n}_v me^{t H_0 }n_v $ so that
  $n_y^{-1} \overline{n}_x = w_0 \overline{n}_{v-y} m e^{t H_0 }n_v$.
  Using Lemma \ref{lemma:cnfnljqwca} again we write $w_0 \overline{n}_{v-y}=\overline{n}_{z}\tilde{m}e^{\tilde{t}H_0}n_z $ with
  \begin{align*}
    z=-\frac{v-y}{\lVert v-y \rVert^2 }=-\frac{-\frac{x}{\lVert x \rVert^2 }-y}{\lVert -\frac{x}{\lVert x \rVert^2 }-y \rVert^2 }=\frac{\lVert x \rVert^2 (x+\lVert x \rVert^2 y)}{\lVert x+\lVert x \rVert^2 y \rVert^2 },\ \tilde{m}= 2 \frac{\left( \frac{x}{\lVert x \rVert^2}+y \right)\left( \frac{x}{\lVert x \rVert^2 }+y \right)^t }{\lVert \frac{x}{\lVert x \rVert^2}+y \rVert^2 }-I_n
  \end{align*}
  and $\tilde{t}= 2 \log \lVert \frac{x}{\lVert x \rVert^2 }+y \rVert$. Thus,
  \begin{align*}
    n_y^{-1} \overline{n}_x = \overline{n}_z \tilde{m}e^{\tilde{t}H_0 }n_zm e^{t H_0 }n_v = \overline{n}_z \tilde{m}m e^{(\tilde{t}+t)H_0 }n_1
  \end{align*}
  for some $n_1 \in N$.
\end{proof}

\subsection{The principal series and intertwining operators}\label{sec:PrincipalSeriesAndIntertwiningOperators}

Identifying $\mathfrak{a}_{\mathbb{C} }^* \cong \mathbb{C} $ by $\lambda\mapsto\lambda(H_0)$, we define, for each $\lambda \in \mathfrak{a}_{\mathbb{C} }^{*}$, a character $e^{\lambda }$ on $A$ by $e^{\lambda }(\exp (t H_0 )) \coloneqq e^{t \lambda }$ for $t \in \mathbb{R} $. For $\sigma \in \widehat{M}$ and $\lambda \in \mathfrak{a}_{\mathbb{C} }^{* }$ we consider the \emph{principal series representation} $\tilde{I}_{\sigma,\lambda}$ induced from the character $\sigma \otimes e^\lambda \otimes 1$ of $P=MAN$, i.e.
\begin{align*}
  \tilde{I}_{\sigma,\lambda}&\coloneqq \mathrm{Ind}_P^G (\sigma \otimes e^\lambda \otimes 1) = \left\{ f \in C^{\infty } (G, V_\sigma )\mid f (gman)= a^{-  (\lambda + \rho )}\sigma (m)^{-1} f (g )\ \forall g \in G,\ man \in P \right\}
\end{align*}
with the left-regular $G$-action $\tilde{\pi }_{\sigma , \lambda }$.

For every $\sigma\in\widehat{M}$ there exists a family of intertwining operators $\tilde{I}_{\sigma,\lambda}\to\tilde{I}_{w_0\sigma,-\lambda}$ depending meromorphically on $\lambda\in\mathbb{C}$. These operators are given by a convergent integral for $\mathrm{Re}(\lambda)>0$ and can be extended to all $\lambda\in\mathbb{C}$ by normalizing with an appropriate gamma factor. To study unitarizability, it is sufficient to consider intertwining operators in the case where $w_0\sigma\simeq\sigma$, and we construct the corresponding family of intertwining operators $\tilde{I}_{\sigma,\lambda}\to\tilde{I}_{\sigma,-\lambda}$ algebraically using the spectrum generating approach of \cite{BransonOlafsson}. We refer to this paper for the details of the construction and only state the results that we need.

We carry out the study of reducibility of $\tilde{I}_{\sigma,\lambda}$ and the construction of intertwining operators in the compact picture. Restriction from $G$ to $K$ defines an isomorphism of $\tilde{I}_{\sigma,\lambda}$ onto
\begin{equation}
	I_{\sigma,\lambda}^{cpt}=\{f\in C^{\infty}(K, V_\sigma )\mid \forall m\in M: f(km)=\sigma(m^{-1} )f(k)\}.
\end{equation}
This isomorphism becomes $G$-equivariant if we endow $I_{\sigma,\lambda}^{cpt}$ with the action
$$ \pi^{cpt}_{\sigma,\lambda}(g)f(k)\coloneqq a(g^{-1} k)^{-(\lambda + \rho)}f(k(g^{-1} k)) \qquad (g\in G,k\in K,f\in I_\sigma^{cpt}(\lambda)), $$
where $a(g)$ resp.\@ $k(g)$ denote the $A$- resp.\@ $K$-component in the Iwasawa decomposition of $g \in G$. The realization on $I_{\sigma,\lambda}^{cpt}$ is called the \emph{compact picture}. The restriction of $\pi_{\sigma,\lambda}^{cpt}$ to $K$ is the left regular representation of $K$ on $I_{\sigma,\lambda}^{cpt}$ and hence independent of $\lambda$. By Frobenius reciprocity and the multiplicity-one property for the pair $(K,M)=(\mathrm{SO}(n+1),\mathrm{SO}(n))$, it decomposes into a multiplicity-free direct sum of $K$-types
$$ I_{\sigma,\lambda}^{cpt} = \widehat{\bigoplus_{\substack{\alpha\in\widehat{K}\\\sigma\preceq\alpha}}}\,\,I_{\sigma,\lambda}^{cpt}(\alpha), $$
where we write $\sigma\preceq\alpha$ if $\sigma$ occurs in $\alpha|_M$. As usual, we parameterize $\alpha$ by its highest weight $(\alpha_1,\ldots,\alpha_m)\in\mathbb{Z}^m$ with $\alpha_1\geq\ldots\geq\alpha_m\geq0$ if $n=2m$ is even and $\alpha_1\geq\ldots\geq\alpha_{m-1}\geq|\alpha_m|$ if $n=2m-1$ is odd. Then $\sigma\preceq\alpha$ if and only if
\begin{alignat}{2}
	&\alpha_1 \geq \sigma_1 \geq \alpha_2 \geq \sigma_2 \geq \ldots \geq \sigma_{m-1} \geq \alpha_m \geq \lvert \sigma_m \rvert \qquad && \text{if }n=2m,\label{eq:InterlacingNeven}\\
	& \alpha_1 \geq \sigma_1 \geq \alpha_2 \geq \sigma_2 \geq \ldots \geq \sigma_{m-1} \geq \lvert \alpha _m \rvert \qquad && \text{if }n=2m-1,\label{eq:InterlacingNodd}
\end{alignat}
where $(\sigma_1,\ldots,\sigma_m)$ resp. $(\sigma_1,\ldots,\sigma_{m-1})$ denotes the highest weight of $\sigma\in\widehat{M}$.

For the construction of intertwining operators, we assume that $\sigma_m=0$ in the case where $n=2m$, then $w_0\sigma\simeq\sigma$ and there exists a meromorphic family of intertwining operators $I_{\sigma,\lambda}^{cpt}\rightarrow I_{\sigma,-\lambda}^{cpt}$. By Schur's Lemma, every such intertwining operator acts by a scalar $a(\alpha)$ on each $K$-type $I_{\sigma,\lambda}^{cpt}(\alpha)$. In \cite[Theorem 3.1]{BransonOlafsson} a meromorphic family of scalars which give rise to intertwining operators is determined:
$$ a(\alpha)=\prod_{k=1}^m\frac{\Gamma(\rho+1-k+\alpha_k-\lambda)}{\Gamma(\rho+1-k+\alpha_k+\lambda)}. $$
Note that for $k=m=\frac{n+1}{2}$ the parameter $\alpha_m$ can be negative, but we can rewrite the corresponding factor
$$ \frac{\Gamma(\rho+1-m+\alpha_m-\lambda)}{\Gamma(\rho+1-m+\alpha_m+\lambda)} = \frac{\Gamma(\frac{1}{2}+\alpha_m-\lambda)}{\Gamma(\frac{1}{2}+\alpha_m+\lambda)} = \frac{\Gamma(\frac{1}{2}+|\alpha_m|-\lambda)}{\Gamma(\frac{1}{2}+|\alpha_m|+\lambda)} $$
by the functional equation for the gamma function. Therefore, we can replace $\alpha_k$ by $|\alpha_k|$ in the formula for $a(\alpha)$. Multiplying $a(\alpha)$ with the meromorphic function
$$ \frac{1}{\Gamma(\rho+\sigma_1-\lambda)}\prod_{k=2}^m\frac{\Gamma(\rho+1-k+\sigma_{k-1}+\lambda)}{\Gamma(\rho+1-k+\sigma_k-\lambda)} $$
leads to the following holomorphic family of scalars
\begin{equation}
	a_{\sigma,\lambda}(\alpha) = \frac{(\rho+\sigma_1-\lambda)_{\alpha_1-\sigma_1}}{\Gamma(\rho+\alpha_1+\lambda)}\prod_{k=2}^m(\rho+1-k+\sigma_k-\lambda)_{|\alpha_k|-\sigma_k}(\rho+1-k+|\alpha_k|+\lambda)_{\sigma_{k-1}-|\alpha_k|}\label{eq:NormalizedEigenvaluesIntertwiners}
\end{equation}
where $(x)_n=\frac{\Gamma(x+n)}{\Gamma(x)}=x(x+1)\cdots(x+n-1)$ denotes the Pochhammer symbol. We let $A_{\sigma,\lambda}:I_{\sigma,\lambda}^{cpt}\to I_{\sigma,-\lambda}^{cpt}$ denote the corresponding intertwining operator.


\subsection{The admissible dual}

In this section we describe the irreducible subrepresentations and quotients of $\tilde{I}_{\sigma,\lambda}$. For the statement it is convenient to put $\sigma_{m+1}=0$ resp. $\sigma_m=0$ in the case where $n=2m$ resp. $n=2m-1$.

\begin{proposition}\label{prop:CompositionSeries}
\begin{enumerate}
	\item The representation $(\pi_{\sigma,\lambda}^{cpt},I_{\sigma,\lambda}^{cpt})$ of $G$ is reducible if and only if $\lambda$ or $-\lambda$ is contained in the set $(\rho-a+\mathbb{N}_0)\setminus\{\rho-k+|\sigma_k|:k=1,\ldots,a\}$, where $a\in\{0,1,\ldots,\lfloor\frac{n}{2}\rfloor\}$ is minimal such that $\sigma_{a+1}=0$.
		\item For $\lambda=\rho+|\sigma_1|+j$ with $j\in\mathbb{N}_0$, there is a unique irreducible subrepresentation $\mathfrak{I}(\sigma,0,j)$ with $K$-types $\{\alpha\succeq\sigma:\alpha_1>\sigma_1+j\}$ and the corresponding quotient $\mathfrak{Q}(\sigma,0,j)$ is also irreducible, finite-dimensional and realized on the $K$-types $\{\alpha\succeq\sigma:\alpha_1\leq\sigma_1+j\}$. If $\sigma_m=0$, then $\mathfrak{I}(\sigma,0,j)$ is the kernel of $A_{\sigma,\lambda}$.
		\item For $\lambda = \rho -i + |\sigma_{i+1}|+j $ with $1\leq i < m$ and $0 \leq j < \sigma_i - |\sigma_{i+1}|$ we have:
  \begin{enumerate}[(i)]
  \item If $i \neq \frac{n-1}{2}$, then there is a unique irreducible subrepresentation $\mathfrak{I}(\sigma,i,j)$ of $I_{\sigma,\lambda}^{cpt}$ with $K$-types $\{\alpha\succeq\sigma:\alpha_{i+1}> |\sigma_{i+1}|+j\}$. The corresponding quotient $\mathfrak{Q}(\sigma ,i, j)$ is also irreducible and realized on the $K$-types $\{\alpha\succeq\sigma:\alpha_{i+1} \leq |\sigma_{i+1}|+j \}$. If $\sigma_m=0$, then $\mathfrak{I}(\sigma,i,j)$ is the kernel of $A_{\sigma,\lambda}$.
  \item If $i= \frac{n-1}{2}$, then the kernel of $A_{\sigma,\lambda}$ is the direct sum of the two irreducible subrepresentations $\mathfrak{I}(\sigma , i, j)^{\pm}$ with $K$-types $\{ \alpha\succeq\sigma: \pm \alpha_{i+1}>j \}$. The corresponding quotient $\mathfrak{Q}(\sigma ,i, j)$ is also irreducible and realized on the $K$-types $\{ \alpha \succeq\sigma:|\alpha_{i+1}|  \leq j\}$.
  \end{enumerate}
\end{enumerate}
\end{proposition}

We summarize the key observations from \cite[(2.13), Section 3.a and Remark 3.2]{BransonOlafsson} that are necessary to prove this result:
\begin{enumerate}[(A)]
	\item\label{eq:BOOobsA} Applying $d\pi_{\sigma,\lambda}^{cpt}(\mathfrak{g})$ to a $K$-type $I_{\sigma,\lambda}^{cpt}(\alpha)$ maps into the direct sum of the $K$-types $I_{\sigma,\lambda}^{cpt}(\beta)$ for $\beta\in\{\alpha\}\cup\{\alpha\pm e_k:k=1,\ldots,m\}$.
	\item\label{eq:BOOobsB} For $\beta=\alpha+e_k$ resp. $\beta=\alpha-e_k$, the $K$-type $I_{\sigma,\lambda}^{cpt}(\beta)$ is in fact contained in the image of $I_{\sigma,\lambda}^{cpt}(\alpha)$ under $d\pi_{\sigma,\lambda}^{cpt}(\mathfrak{g})$ if and only if
	\begin{equation}
		c_k^+(\lambda,\alpha)=\rho+1-k+\alpha_k+\lambda \neq 0 \qquad \mbox{resp.} \qquad c_k^-(\lambda,\alpha)=\rho-k+\alpha_k-\lambda \neq 0.\label{eq:DefinitionCplusminus}
	\end{equation}
	\item\label{eq:BOOobsC} Let $V \subseteq I_{\sigma,\lambda}^{cpt}$ be a subrepresentation. If $n=2m$ and $\sigma_m\neq0$, then there is no non-trivial intertwining operator $V\to I_{\sigma,-\lambda}^{cpt}$. In all other cases a map $A:V\to I_{\sigma,-\lambda}^{cpt}$ is intertwining if and only if it is given by a scalar $a(\alpha)$ on each $K$-type $I_{\sigma,\lambda}^{cpt}(\alpha)\subseteq V$ and the scalars $\{a(\alpha)\}_\alpha$ satisfy the following recurrence relation whenever both $\alpha$ and $\alpha+e_k$ are $K$-types of $V$:
	\begin{equation*}
		(\rho+1-k+\alpha_k+\lambda)\cdot a(\alpha+e_k) =(\rho+1-k+\alpha_k-\lambda)\cdot a(\alpha).
	\end{equation*}
\end{enumerate}

Observation \eqref{eq:BOOobsB} implies in particular that $\pi_{\sigma,\lambda}$ can only be reducible if $\lambda-\rho\in\mathbb{Z}$. Since $\pi_{\sigma,\lambda}$ and $\pi_{\sigma^*,-\lambda}$ are dual to each other, it is sufficient to consider the case $\lambda\geq0$. We treat the cases $\lambda\geq\rho+|\sigma_1|$ and $0\leq\lambda<\rho+|\sigma_1|$ separately.


\begin{lemma}\label{lem:CompositionSeries1}
  Let $\lambda = \rho + |\sigma_1| +j $ for some $j \in \mathbb{N}_0$. Then $I_{\sigma,\lambda}^{cpt}$ has a unique irreducible subrepresentation $\mathfrak{I}(\sigma ,0, j)$ realized on the $K$-types $\{\alpha\succeq\sigma:\alpha_1>|\sigma_1|+j\}$ and the quotient $\mathfrak{Q}(\sigma ,0, j)=I_{\sigma,\lambda}^{cpt}/\mathfrak{I}(\sigma ,0, j)$ is irreducible and finite-dimensional with $K$-types $\{\alpha\succeq\sigma:\alpha_1\leq|\sigma_1|+j\}$. Moreover, if $\sigma_m=0$ then $\mathfrak{I}(\sigma ,0, j)$ is the kernel of the intertwining operator $A_{\sigma,\lambda}$ and $\mathfrak{Q}(\sigma , 0,j)$ is isomorphic to its image.
\end{lemma}

\begin{proof}
	Using the inequalities in \eqref{eq:InterlacingNeven} and \eqref{eq:InterlacingNodd} one can show that $c_k^+(\lambda,\alpha)>0$ for all $k=1,\ldots,m$ and all $\alpha\succeq\sigma$. Moreover,
	$$ c_k^-(\lambda,\alpha) = -k+\alpha_k-|\sigma_1|-j \begin{cases}\leq -k-j<0&\mbox{for $k>1$,}\\=-1+\alpha_1-|\sigma_1|-j&\mbox{for $k=1$.}\end{cases} $$
	It follows that $c_k^-(\lambda,\alpha)=0$ if and only if $k=1$ and $\alpha_1=|\sigma_1|+j+1$. This shows the first claim. For the second claim it suffices to observe that the only factor in \eqref{eq:NormalizedEigenvaluesIntertwiners} which can vanish is $(\rho+\sigma_1-\lambda)_{\alpha_1-\sigma_1}=(-j)_{\alpha_1-\sigma_1}$, and it vanishes if and only if $\alpha_1-\sigma_1>j$.
\end{proof}

We now discuss the case $0\leq\lambda<\rho+|\sigma_1|$.


\begin{lemma}\label{lemma:cngyxhhv0u}
  Let $0 \leq  \lambda < \rho + |\sigma_1| $. Then $I_{\sigma,\lambda}^{cpt}$ is reducible if and only if $\lambda = \rho -i + |\sigma_{i+1}|+j $ for some $1\leq i \leq m-1$ with $0 \leq j < \sigma_i - |\sigma_{i+1}|$.
  \begin{enumerate}[(i)]
  \item If $i \neq \frac{n-1}{2}$, then $I_{\sigma,\lambda}^{cpt}$ has a unique irreducible subrepresentation $\mathfrak{I}(\sigma ,i, j)$ realized on the $K$-types $\left\{\alpha\succeq\sigma \mid \alpha_{i+1}> |\sigma_{i+1}|+j \right\}$. The corresponding quotient $\mathfrak{Q}(\sigma , i, j)$ is also irreducible and realized on the $K$-types $\left\{ \alpha\succeq\sigma: \alpha_{i+1} \leq |\sigma_{i+1}|+j \right\}$. Moreover, if $\sigma_m=0$ then $\mathfrak{I}(\sigma ,i, j)$ is the kernel of the intertwining operator $A_{\sigma,\lambda}$ and $\mathfrak{Q}(\sigma , i, j)$ is isomorphic to its image.
    \item If $i= \frac{n-1}{2}$, then $I_{\sigma,\lambda}^{cpt}$ has two irreducible subrepresentations $\mathfrak{I}(\sigma , \frac{n-1}{2}, j)^{\pm}$ realized on the $K$-types $\{ \alpha\succeq\sigma:\pm \alpha_{i+1}>j \}$. The corresponding quotient $\mathfrak{Q}(\sigma ,i, j)$ is also irreducible and realized on the $K$-types $\left\{ \alpha\succeq\sigma: |\alpha_{i+1}|  \leq j\right\}$. Moreover, the direct sum $\mathfrak{I}(\sigma , \frac{n-1}{2}, j)^+\oplus\mathfrak{I}(\sigma , \frac{n-1}{2}, j)^-$ is the kernel of the intertwining operator $A_{\sigma,\lambda}$ and $\mathfrak{Q}(\sigma ,i, j)$ is isomorphic to its image.
    \end{enumerate}
\end{lemma}

\begin{proof}
	We first note that for $0\leq\lambda<\rho+|\sigma_1|$ and $\alpha\succeq\sigma$ with $\alpha+e_1\succeq\sigma$ resp. $\alpha-e_1\succeq\sigma$ we have
	$$ c_1^+(\lambda,\alpha) = \rho+\alpha_1+\lambda \geq \rho > 0 \qquad \mbox{resp.} \qquad c_1^-(\lambda,\alpha) = \rho-1+\alpha_1-\lambda > -1+\alpha_1-|\sigma_1| > 0, $$
	so that there is no reducibility in the $e_1$-direction. We therefore let $1<k\leq m=\lfloor\frac{n+1}{2}\rfloor$ and determine for which $\lambda$ and $\alpha$ one cannot reach the $K$-type $\alpha\pm e_k$ from $\alpha$. Starting with $\alpha+e_k$, we find
	$$ c_k^+(\lambda,\alpha) = \rho+1-k+\alpha_k+\lambda \geq \rho+1-k+\alpha_k \geq \frac{1}{2}+\alpha_k, $$
	so we can only have $c_k^+(\lambda,\alpha)=0$ if $\alpha_k$ is negative. This is only possible if $n=2m-1$ is odd and $k=m$, and in this case $c_m^+(\lambda,\alpha)=0$ if and only if $\lambda=-\rho-1+m-\alpha_m=-\frac{1}{2}-\alpha_m$. This produces the points of reducibility $\lambda=\frac{1}{2}+j=\rho-(m-1)+j$ with $0\leq j<\sigma_{m-1}$ and the corresponding subrepresentation has $K$-types $\{\alpha\succeq\sigma:\alpha_m<-j\}$. On the other hand, we cannot reach $\alpha-e_k$ from $\alpha$ if $c_k^-(\lambda,\alpha)=\rho-k+\alpha_k-\lambda=0$. This only happens if $\lambda=\rho-k+\alpha_k$, where $\sigma_{k-1}\geq\alpha_k>|\sigma_k|$ (since $\alpha-e_k$ is a $K$-type as well). Writing $k=i+1$ we obtain the points of reducibility $\lambda=\rho-i+|\sigma_{i+1}|+j$ with $0\leq j<\sigma_i-|\sigma_{i+1}|$ and the corresponding subrepresentation has $K$-types $\{\alpha\succeq\sigma:\alpha_{i+1}>|\sigma_{i+1}|+j\}$. With the convention that $\sigma_m=0$ for $n=2m-1$ odd this implies the claims about subrepresentations and reducibility.\\
	We now show the claims about the kernel of the intertwining operator $A_{\sigma,\lambda}$ using the explicit expression \eqref{eq:NormalizedEigenvaluesIntertwiners} for its eigenvalues. For $0\leq\lambda<\rho+|\sigma_1|$ clearly the factors $(\rho+|\sigma_1|-\lambda)_{\alpha_1-|\sigma_1|}$ and $\Gamma(\rho+\alpha_1+\lambda)^{-1}$ are non-zero.  Now let $\lambda=\rho-i+|\sigma_{i+1}|+j$ with $1\leq i<m$ and $0\leq j<\sigma_i-|\sigma_{i+1}|$, then
	\begin{align*}
		(\rho+1-k+\sigma_k-\lambda)_{|\alpha_k|-\sigma_k} &= (i+1-k+\sigma_k-|\sigma_{i+1}|-j)_{|\alpha_k|-\sigma_k},\\
		(\rho+1-k+|\alpha_k|+\lambda)_{\sigma_{k-1}-|\alpha_k|} &= (2\rho+1-i-k+|\alpha_k|+\sigma_{i+1}+j)_{\sigma_{k-1}-|\alpha_k|}.
	\end{align*}
	The second term is always non-zero since $i<m$, $k\leq m$, $j\geq0$ together with \eqref{eq:InterlacingNeven} for $n=2m$ even imply
	$$ 2\rho+1-i-k+|\alpha_k|+\sigma_{i+1}+j = n+1-i-k+|\alpha_k|+\sigma_{i+1}+j > n+1-2m \geq 0. $$
	The first term vanishes if and only if
	$$ i+1-k+\sigma_k-|\sigma_{i+1}|-j \leq 0 \qquad \mbox{and} \qquad i+1-k+|\alpha_k|-|\sigma_{i+1}|-j > 0. $$
	For $i+1>k$ both terms are $>0$ and for $i+1<k$ both are $<0$. For $i+1=k$ the first term is equal to $0$ (hence in particular $\leq0$) and the second term is equal to $|\alpha_{i+1}|-\sigma_{i+1}-j$ which is $>0$ if and only if $|\alpha_{i+1}|>\sigma_{i+1}+j$. This shows that the kernel of $A_{\sigma,\lambda}$ consists of the $K$-types $\{\alpha\succeq\sigma:|\alpha_{i+1}|>\sigma_{i+1}+j\}$, so it equals $\mathfrak{I}(\sigma,i,j)$ if $i\neq\frac{n-1}{2}$ and $\mathfrak{I}(\sigma,i,j)^+\oplus\mathfrak{I}(\sigma,i,j)^-$ if $i=\frac{n-1}{2}$.
\end{proof}

\begin{proof}[Proof of Proposition~\ref{prop:CompositionSeries}]
Combining Lemma~\ref{lem:CompositionSeries1} and \ref{lemma:cngyxhhv0u} shows the statements.
\end{proof}

\subsection{The unitary dual}

We now turn to the question of unitarity. For this, we first determine which of the irreducible subrepresentations and quotients are unitarizable. For the composition factors from Lemma~\ref{lem:CompositionSeries1} we find:

\begin{lemma}\label{lem:Unitarity1}
	$\mathfrak{I}(\sigma ,0, j)$ is unitarizable if and only if $\sigma$ is trivial. $\mathfrak{Q}(\sigma ,0,j)$ is unitarizable if and only if it is the trivial representation, i.e., if $\sigma = 0$ and $j = 0$.
\end{lemma}

\begin{proof}
	The existence of an invariant inner product on a subrepresentation $V\subseteq I_{\sigma,\lambda}^{cpt}$ is equivalent to the existence of an intertwining operator $A:V\to I_{\sigma,-\lambda}^{cpt}$ whose eigenvalues $a(\alpha)$ on the $K$-types $I_{\sigma,\lambda}(\alpha)\subseteq V$ are all strictly positive (see e.g. \cite[p.\@ 203]{BransonOlafsson}). Let us first consider $\mathfrak{I}(\sigma,0,j)$. If both $\alpha$ and $\alpha+e_k$ are $K$-types and $k\geq2$, we have by Observation~\ref{eq:BOOobsC}
	\begin{equation*}
		\frac{a(\alpha+e_k)}{a(\alpha)} = \frac{1-k+ \alpha_k - \sigma_1 -j}{n+1-k+ \alpha_k + \sigma_1 +j}<0
	\end{equation*}
	since $\sigma \preceq \alpha $. Thus, every possible intertwining operator defined on $\mathfrak{I}(\sigma ,0, j)$ changes signs between $\alpha$ and $\alpha+e_k$, so $\mathfrak{I}(\sigma ,0, j)$ is not unitarizable in these cases. This situation never occurs if and only if either $n=2m$ and $\sigma_1=\ldots=\sigma_m$ or if $n=2m-1$ and $\sigma = (0,\ldots,0)$. For $n=2m$ there only exist intertwining operators if $\sigma_m=0$ by Observation~\ref{eq:BOOobsC}, so for both even and odd $n$ we can restrict to the case where $\sigma=(0,\ldots,0)$ in which there exists a non-trivial holomorphic family $A_{\sigma,\lambda}$ of intertwining operators (see Section~\ref{sec:PrincipalSeriesAndIntertwiningOperators}) and hence, by regularization, a non-zero operator $\mathfrak{I}(\sigma,0,j)\to I^{cpt}_{\sigma,-\lambda}$. In this case, only $K$-types of the form $\alpha =(j+1+ \ell)e_1$, $\ell \in \mathbb{N}_0$, occur in $\mathfrak{I}(\sigma,0,j)$ and we have
	\begin{equation*}
		\frac{a((j+2+\ell)e_1)}{a((j+1+\ell)e_1)} = \frac{1+\ell}{n+2j+1+\ell} > 0 \qquad \mbox{for all $\ell\geq0$,}
	\end{equation*}
	so that $\mathfrak{I}(\sigma , 0,j)$ is unitarizable in this case. Finally note that $\mathfrak{Q}(\sigma ,0, j)$ is finite dimensional, so it is unitarizable if and only if it is the trivial representation, i.e., if $\sigma = 0$ and $j=0$.
\end{proof}

Next, we consider the composition factors from Lemma~\ref{lemma:cngyxhhv0u}:

\begin{lemma}\label{lem:Unitarity2}
  Let $\lambda = \rho -i + |\sigma_{i+1}|+j$ as in Lemma \ref{lemma:cngyxhhv0u} and $a=\min\{k:\sigma_{k+1}=0\}=\max\{k:\sigma_k\neq\sigma_{k+1}\}$.
  \begin{enumerate}[(i)]
  \item $\mathfrak{I}(\sigma , i, j)$ is unitarizable if and only if $\sigma_m=0$ and $i=a$,
  \item $\mathfrak{Q}(\sigma ,i ,j)$ is unitarizable if and only if $\sigma_m=0$, $i=a$ and $j=0$,
  \item $\mathfrak{I}(\sigma , \frac{n-1}{2}, j)^{\pm}$ is unitarizable for all $j\geq0$.
  \end{enumerate}
\end{lemma}

\begin{proof}
	First note that by Observation~\eqref{eq:BOOobsC} we may assume $\sigma_m=0$ if $n=2m$. Using the same method as in the proof of Lemma~\ref{lem:Unitarity1}, we have to check whether
	\begin{equation*}
		\frac{a(\alpha+e_k)}{a(\alpha)} = \frac{\rho+1-k+\alpha_k-\lambda}{\rho+1-k+\alpha_k+\lambda} > 0
	\end{equation*}
	whenever both $\alpha$ and $\alpha + e_k$ are contained in the relevant composition factor. For the numerator we have
  \begin{equation}\label{eq:cngyxxq8xg}
    \rho+1-k+\alpha_k-\lambda = -(k-1)+i+\alpha_k - |\sigma_{i+1}|-j\begin{cases}
      >0 & \text{if } i=k-1 \text{ and } \alpha_k > |\sigma_{i+1}|+j,\\
      <0 & \text{if } i=k-1 \text{ and } \alpha_k < |\sigma_{i+1}|+j,\\
      >0 & \text{if } i > k-1,\\
      <0 & \text{if } i < k-1,
    \end{cases}
  \end{equation}
  since, if $i < k-1$ then $\sigma_{i+1} \geq \alpha_k$, and if $i > k-1$ then $\alpha_k \geq \sigma_i > |\sigma_{i+1} | +j$. Moreover, for the denominator we find
  \begin{equation}\label{eq:cngyxvbgbm}
    \rho+1-k+\alpha_k+\lambda = n-i-(k-1)+ \alpha_k +| \sigma_{i+1}|+j
    \begin{cases}
      <0      & \text{if } i=k-1= \frac{n-1}{2} \text{ and } -\alpha_k > |\sigma_{i+1}|+j+1,\\
      >0      & \text{else},
    \end{cases}
  \end{equation}
  since $\alpha_k + |\sigma_{i+1}|+j \geq 0$ if $k \neq \frac{n+1}{2}$. For $\mathfrak{I}(\sigma , i, j)$ we have that \eqref{eq:cngyxvbgbm} is always positive since $i \neq \frac{n-1}{2}$. To ensure that also \eqref{eq:cngyxxq8xg} is positive for all $\alpha ,\, \alpha + e_k$ occurring in $\mathcal{I}(\sigma , i, j)$ we have to ensure that $i<k-1$ is not possible. This is the case if and only if $i$ is maximal such that $\sigma_i\neq\sigma_{i+1}$. The other cases follow from similar considerations.
\end{proof}

Finally, we determine the possible complementary series:

\begin{lemma}\label{lem:ComplementarySeries}
  Let $\sigma \in \widehat{M}$ and $\lambda\in\mathbb{R}$. Then $I_{\sigma,\lambda}^{cpt}$ is irreducible and unitarizable (i.e. belongs to the complementary series) if and only if $\sigma_m=0$ and
  \begin{equation*}
    \lvert \lambda \rvert < \rho - a,
  \end{equation*}
  where $a=\max (\left\{ k\mid \sigma_k \neq \sigma_{k+1} =0 \right\} \cup \{0\}) = \min\{k\mid \sigma_{k+1}=0\}$.
\end{lemma}

\begin{proof}
  By \cite[La.\@ 5.2]{BransonOlafsson} we have complementary series representations for $\sigma_m=0$ and
  $$ |\lambda| < \min\{c_k^\pm(0,\alpha):\alpha,\alpha\pm e_k\succeq\sigma\}. $$
  By \eqref{eq:DefinitionCplusminus} we thus have to minimize
  \begin{equation*}
    \lvert \rho-k+\alpha_k+1 \rvert
  \end{equation*}
  for $\sigma \preceq \alpha + e_k, \alpha $ and $1 \leq k \leq m=\lfloor \frac{n+1}{2} \rfloor$. In all cases, this expression is minimal for $k=a+1$ and $\alpha_k=0$.
\end{proof}

We collect the results about unitarity in the following theorem. For simplicity, we put $\sigma_0=\infty$, then the unitarity results for $i=0$ and $i>0$ can be combined.

\begin{theorem}
	For $\sigma\in\widehat{M}$ let $a_\sigma=\min\{k:\sigma_{k+1}=0\}$. The unitary dual of $G$ consists of the following representations:
	\begin{itemize}
		\item The unitary principal series $I_{\sigma,\lambda}$ with $\sigma\in\widehat{M}$ and $\lambda\in i\mathbb{R}$.
		\item The complementary series $I_{\sigma,\lambda}$ with $\sigma\in\widehat{M}$, $\sigma_m=0$ and $0<\lambda<\rho-a_\sigma$.
		\item The unitarizable subrepresentations $\mathfrak{I}(\sigma,a_\sigma,j)$ resp. $\mathfrak{I}(\sigma,a_\sigma,j)^\pm$ with $\sigma_m=0$, $a_\sigma\neq\frac{n+1}{2}$ resp. $a_\sigma=\frac{n+1}{2}$ and $0\leq j<\sigma_{a_\sigma}$.
		\item The unitarizable quotients $\mathfrak{Q}(\sigma,a_\sigma,0)$ with $\sigma\in\widehat{M}$ and $\sigma_m=0$.
	\end{itemize}
\end{theorem}

\begin{proof}
	It follows from Observation~\ref{eq:BOOobsB} that the unitary principal series $\pi_{\sigma,\lambda}^{cpt}$, $\lambda\in i\mathbb{R}$, is always irreducible, and it is unitary with respect to the $L^2$-inner product on $L^2(K)\otimes V_\sigma$. Together with the complementary series (see Lemma~\ref{lem:ComplementarySeries}), these are all irreducible unitarizable principal series representations. The remaining representations in the unitary dual are to be found among the composition factors at points of reducibility $\lambda\in\mathbb{R}\setminus\{0\}$. By \cite[Remarks 3.2 and 3.3]{BransonOlafsson}, composition factors in the case where $n=2m$ is even and $\sigma_m\neq0$ are never unitary. The unitary ones in the other cases are determined in Lemma~\ref{lem:Unitarity1} and \ref{lem:Unitarity2} for $\lambda>0$, and since $\pi_{\sigma^*,-\lambda}$ is dual to $\pi_{\sigma,\lambda}$ these are all.
\end{proof}

\begin{remark}
  Comparing $K$-types and infinitesimal characters, we find the following equivalences:
  \begin{itemize}
  \item $\mathfrak{I}(\sigma , i, j) \cong \mathfrak{Q}(\sigma +(j+1)e_{i+1}, i+1, \sigma_{i+1}-\lvert \sigma_{i+2} \rvert)$ if $0\leq i < m-1$,
  \item $\mathfrak{I}(\sigma , m-1,j) \cong \begin{cases}I_{\sigma + \mathrm{sgn}(\sigma_m)(j+1)e_m,\sigma_m}^{cpt}&\mbox{if $n=2m$ is even and $\sigma_m\neq0$,}\\I_{\sigma +(j+1)e_m,0}^{cpt} \cong I_{\sigma -(j+1)e_m,0}^{cpt}&\mbox{if $n=2m$ is even and $\sigma_m=0$.}\end{cases}$
  \end{itemize}
\end{remark}

\begin{example}
  For $n=3$ we obtain the following composition structure of $I_{\sigma,\lambda}^{cpt}$ for $\lambda\geq0$ and $\sigma = \sigma_1 \in \mathbb{N}_0$: We only have reducibility if $\lambda \in(\frac{1}{2} + \mathbb{N}_0) \setminus \left\{ \frac{1}{2} + \sigma \right\}$. When $\lambda = \frac{3}{2} + \sigma + j$ for some $j \in \mathbb{N}_0$, then $\mathfrak{I}(\sigma,0,j)$ and $\mathfrak{Q}(\sigma,0,j)$ are irreducible and realized on the $K$-types
  \begin{equation*}
	\left\{ (\alpha_1 , \alpha_2) \in \widehat{K} \mid \alpha_1 > \sigma + j = \lambda - \frac{1}{2}\right\} \quad \text{ resp. } \quad \left\{ (\alpha_1 , \alpha_2) \in \widehat{K} \mid \alpha_1 \leq \lambda - \frac{1}{2} \right\}.
\end{equation*}
If $\lambda = \frac{1}{2} + j$ for some $0 \leq j < \sigma$, then $\mathfrak{I}(\sigma , 1,j)^{\pm}$ and $\mathfrak{Q}(\sigma , 1,j)$ are irreducible and realized on the $K$-types
\begin{equation*}
	\left\{ (\alpha_1 , \alpha_2) \in \widehat{K} \mid \pm \alpha_2 > j = \lambda - \frac{1}{2}\right\} \text{ resp. } \left\{(\alpha_1 , \alpha_2) \in \widehat{K}\mid  |\alpha_2| \leq \lambda - \frac{1}{2} \right\}.
\end{equation*}
\end{example}

\section{The non-compact picture and its Fourier transform}

By restricting functions on $G$ to the open dense Bruhat cell $\overline{N}P\subseteq G$ we obtain a realization of the principal series on a space of functions on $\overline{N}\simeq\mathbb{R}^n$, the so-called \emph{non-compact picture}. Taking the Euclidean Fourier transform on $\mathbb{R}^n$ gives yet another realization and we study the Lie algebra action in this realization.

\subsection{The non-compact picture}

Since $\overline{N}P\subseteq G$ is open and dense and $\overline{N}\cap P=\{1\}$, the restriction map $\tilde{I}_{\sigma,\lambda} \to C^\infty (\overline{N}, V_\sigma ), \ f \mapsto f_{\overline{N}}\coloneqq f|_{\overline{N}}$
is one-to-one. We denote its image by $I_{\sigma,\lambda}$. Letting $\pi_{\sigma , \lambda } (g)f_{\overline{N}} \coloneqq (\tilde{\pi }_{\sigma , \lambda }(g)f)_{\overline{N}}$ we obtain a $G$-representation $\pi_{\sigma,\lambda}$ on $I_{\sigma,\lambda}$. In the following we identify $\overline{N}\cong \mathbb{R}^n $ by $\overline{n}_x \mapsto x$ and abbreviate $f (\overline{n}_x )=f (x)$ for $x \in \mathbb{R}^n $ and $f \in I_{\sigma,\lambda}$.

\begin{lemma}\label{lem:NonCptPictureTempered}
For every $\lambda\in\mathbb{C}$, we have
$$ \mathcal{S}(\mathbb{R}^n)\otimes V_\sigma \subseteq I_{\sigma,\lambda} \subseteq C^\infty_{\mathrm{temp}}(\mathbb{R}^n)\otimes V_\sigma, $$
where $\mathcal{S}(\mathbb{R}^n)$ is the Schwartz space and $C^\infty_{\mathrm{temp}}(\mathbb{R}^n)$ the space of tempered smooth functions on $\mathbb{R}^n$. More precisely, every $f\in I_{\sigma,\lambda}$ grows at most of order $(1+\lVert x \rVert^2)^{-(\lambda+\rho)}$ as $x\to\infty$.
\end{lemma}

\begin{proof}
For the first inclusion we extend $f\in\mathcal{S}(\mathbb{R}^n)\otimes V_\sigma$ to $\widetilde{f}\in \tilde{I}_{\sigma,\lambda}$ by
$$ \widetilde{f}(g) \coloneqq \begin{cases}a^{-\lambda-\rho}\sigma(m)^{-1}f(x)&\mbox{if }g=\overline{n}_xman\in\overline{N}MAN,\\0&\mbox{else.}\end{cases} $$
Since $f$ vanishes at infinity together with all its derivatives, this defines a smooth function $\widetilde{f}$ on $G$ which clearly has the right equivariance properties. This shows the first inclusion. For the second inclusion we first decompose $\overline{n}_x=kan\in KAN$ in terms of the Iwasawa decomposition. Then
$$ n_x^{-1}\overline{n}_x = \theta(\overline{n}_x)^{-1}\overline{n}_x = \theta(n)^{-1}\theta(a)^{-1}\theta(k)^{-1}kan = \theta(n)^{-1}a^2n $$
since $\theta(k)=k$, $\theta(a)=a^{-1}$ and $\theta(\overline{n}_x)=n_x$. By Lemma~\ref{lemma:cnfn1kagf2} we find that $a^2=\exp tH_0$ with $t=2\log(1+\|x\|^2)$, so that
$$ f_{\overline{N}}(x) = f(\overline{n}_x) = f(kan) = a^{-\lambda-\rho}f(k) = (1+\|x\|^2)^{-(\lambda+\rho)}f(k) \qquad (f\in\tilde{I}_{\sigma,\lambda}). $$
Since $K$ is compact, $f(K)$ is bounded and the claim follows.
\end{proof}

The action of $\overline{N}, M$ and $A$ with respect to $\pi_{\sigma , \lambda }$ is given as follows:
\begin{align*}
  \pi_{\sigma , \lambda }(\overline{n}_y )f (x) &= f (x-y) && y\in\mathbb{R}^n,\\
  \pi_{\sigma , \lambda } (m)f (x) &= \sigma (m)f (m^{-1} x) && m\in M\simeq\mathrm{SO}(n),\\
  \pi_{\sigma , \lambda }(e^{t H_0 })f (x) &= e^{(\lambda + \rho )t}f (e^t x) && t\in\mathbb{R}.
\end{align*}
Using Lemma \ref{lemma:cnfn1kagf2} we also obtain a description of the $N$-action, which gives us the following expressions for the derived representation $d \pi_{\sigma , \lambda }$ of $\pi_{\sigma , \lambda }$:
\begin{alignat}{2}
  \label{eq:NpictureActionNbar}d \pi_{\sigma , \lambda }(\overline{N}_j )f (x)&=-\partial_j f (x),\hfill &&\quad j=1,\ldots , n\\
  \label{eq:cnfpvp1t7g}d \pi_{\sigma , \lambda }(T)f (x)&=d \sigma (T)f (x)-D_{Tx}f (x),\hfill &&\quad T \in \mathfrak{m}\cong \mathfrak{so}(n)\\
  \label{eq:NpictureActionA}d \pi_{\sigma , \lambda }(H_0 )f (x)&= (E+\lambda + \rho )f (x),\hfill &&\\
  \label{align:cnfn2gqhhw}d \pi_{\sigma , \lambda } (N_j )f (x) &= \lVert x \rVert^2 \partial_j f(x)-2 x_j (E + \lambda + \rho )f (x)+2 d \sigma (x e_j^t - e_j x^t )f (x),\hfill &&\quad j=1,\ldots , n
\end{alignat}
where $D_a $ denotes the directional derivative in the direction of $a \in \mathbb{R}^n $ and $E \coloneqq \sum_{j=1}^{n}x_j \partial_j $ denotes the Euler operator on $\mathbb{R}^n $.

\subsection{The F-picture}

By Lemma~\ref{lem:NonCptPictureTempered}, we have $I_{\sigma,\lambda}\subseteq C^\infty_{\mathrm{temp}}(\mathbb{R}^n)\otimes V_\sigma\subseteq\mathcal{S}'(\mathbb{R}^n)\otimes V_\sigma$, so the Euclidean Fourier transform is defined on $I_{\sigma,\lambda}$. We normalize the Fourier transform by
\begin{align*}
	\mathcal{F}(f)(\xi ) \coloneqq  \frac{1}{(2 \pi )^{\frac{n}{2}}} \int_{\mathbb{R}^n}e^{- i \langle x, \xi  \rangle}f (x) \,\mathrm{d}x.
\end{align*}
For $\lambda \in \mathfrak{a}_{\mathbb{C} }^{*} $ and $\sigma \in \widehat{M}$ we define a representation $\widehat{\pi}_{\sigma , \lambda }$ of $G$ on $\mathcal{F}(I_{\sigma,\lambda})$ by
\begin{align*}
\widehat{\pi}_{\sigma , \lambda }(g) \circ \mathcal{F} \coloneqq \mathcal{F} \circ \pi_{\sigma , \lambda }(g), \quad g \in G.
\end{align*}
Again by Lemma~\ref{lem:NonCptPictureTempered} we have the inclusions
\begin{equation}
	\mathcal{S}(\mathbb{R}^n)\otimes V_\sigma \subseteq \mathcal{F}(I_{\sigma,\lambda}) \subseteq \mathcal{S}'(\mathbb{R}^n)\otimes V_\sigma.\label{eq:FPictureSchwartz}
\end{equation}

By the intertwining properties of the Fourier transform, we obtain the following formulas for the action of $\overline{P}$:
\begin{align}
	\widehat{\pi}_{\sigma,\lambda}(\overline{n}_x)f(\xi) &= e^{-i\langle x,\xi\rangle}f(\xi) && (x\in\mathbb{R}^n)\label{eq:ActionFpictureNbar},\\
	\widehat{\pi}_{\sigma,\lambda}(m)f(\xi) &= \sigma(m)f(m^{-1}\xi) && (m\in M),\label{eq:ActionFpictureM}\\
	\widehat{\pi}_{\sigma,\lambda}(e^{tH_0})f(\xi) &= e^{(\lambda-\rho)t}f(e^{-t}\xi) && (t\in\mathbb{R}).\label{eq:ActionFpictureA}
\end{align}
This implies that the derived representation $d\widehat{\pi}_{\sigma,\lambda}$ of $\widehat{\pi}_{\sigma,\lambda}$ is on $\overline{\mathfrak{n}}$, $\mathfrak{m}$ and $\mathfrak{a}$ given by
\begin{alignat}{2}
  d\widehat{\pi}_{\sigma , \lambda }(\overline{N}_j )f (\xi )&=-i \xi_j f (\xi ),&& \qquad j=1,\ldots , n\label{eq:FpictureActionNbar}\\
  d\widehat{\pi}_{\sigma , \lambda }(T)f (\xi )&= d \sigma (T)f (\xi )- D_{T \xi }f (\xi ),&& \qquad T \in \mathfrak{m} \cong \mathfrak{so}(n)\label{eq:FpictureActionM}\\
  d \widehat{\pi}_{\sigma , \lambda }(H_0 )f (\xi )&= - (E-\lambda +\rho )f (\xi ).\label{eq:FpictureActionA}
\end{alignat}
The $\mathfrak{n}$-action is given by the following lemma.

\begin{lemma}\label{lem:FpictureActionN}
  For each $j \in \{1, \ldots , n\}$ we have
  \begin{align*}
d \widehat{\pi}_{\sigma , \lambda }(N_j )f (\xi )= -i \Big(\xi_j \Delta -2 (E-\lambda +\rho )\partial_j -2d \sigma (\partial e_j^t - e_j \partial^t )\Big)f (\xi ),
\end{align*}
where $\partial \coloneqq \sum_{k=1}^{n}\partial_k e_k $.
\end{lemma}

\begin{proof}
  By Equation \eqref{align:cnfn2gqhhw} and the properties of the Fourier transform
  \begin{equation*}
  	\partial_j \circ \mathcal{F} = \mathcal{F} \circ (-i x_j ) \qquad \mbox{and} \qquad \xi_j \circ \mathcal{F} = \mathcal{F} \circ (-i \partial_j ),
  \end{equation*}
  we have
  \begin{align*}
    d \widehat{\pi}_{\sigma , \lambda }(N_j )f (\xi)&=-i \Delta  \xi _j f (\xi )-2i\partial_j (\lambda + \rho - \sum_{k=1}^{n}\partial_k \xi _k )f (\xi )+2i d \sigma ( \partial e_j^t - e_j  \partial^t ) f (\xi )\\
    &=-i \xi _j \Delta f (\xi )-2i \partial_j f (\xi )-2i \partial_j (\lambda + \rho )f (\xi )+2i \partial_j (E+n)f (\xi )+2i d \sigma (\partial e_j^t - e_j \partial^t ) f (\xi )\\
    &=-i \xi _j \Delta f (\xi )-2i \partial_j f (\xi )-2i \partial_j (\lambda - \rho )f (\xi )+2i (E \partial_j + \partial_j )f (\xi )+2i d \sigma (\partial e_j^t - e_j \partial^t )f (\xi ).\qedhere
\end{align*}
\end{proof}

In what follows we abbreviate
\begin{align*}
\mathcal{B}_{\lambda ,j}^{\sigma }\coloneqq \xi_j \Delta -2 (E-\lambda + \rho )\partial_j - 2 d \sigma (\partial e_j^t - e_j \partial^t ).
\end{align*}

\subsection{Decomposing the \texorpdfstring{$N$}{N}-action}

In this section we investigate the $N$-action in the F-picture in more detail. For this, we first restrict functions to $\mathbb{R}^n\setminus\{0\}$.

\begin{lemma}\label{lem:KernelRestriction}
The kernel of the restriction map $\mathcal{F}(I_{\sigma,\lambda})\to\mathcal{D}'(\mathbb{R}^n\setminus\{0\})\otimes V_\sigma$ is the largest finite-dimensional subrepresentation of $\mathcal{F}(I_{\sigma,\lambda})$.
\end{lemma}

\begin{proof}
We consider the composition of the restriction with the Fourier transform $I_{\sigma,\lambda}\to\mathcal{F}(I_{\sigma,\lambda})\to\mathcal{D}'(\mathbb{R}^n\setminus\{0\})\otimes V_\sigma$. The kernel of this map consists of all polynomials in $I_{\sigma,\lambda}$. If $f\in I_{\sigma,\lambda}$ is a polynomial, then the subrepresentation $d\pi_{\sigma,\lambda}(\mathcal{U}(\mathfrak{g}))f$ of $I_{\sigma,\lambda}$ generated by $f$ is contained in the space of polynomials by \eqref{eq:NpictureActionNbar}--\eqref{align:cnfn2gqhhw}. But by Lemma~\ref{lem:NonCptPictureTempered} the degree of every such polynomial is at most $-2(\lambda+\rho)$, so this subrepresentation is finite-dimensional. If conversely $V\subseteq I_{\sigma,\lambda}$ is a finite-dimensional subrepresentation, then the nilpotency of $\overline{\mathfrak{n}}$ implies that $d\pi_{\sigma,\lambda}(\mathfrak{n})^kV=\{0\}$ for sufficiently large $k$. But $d\pi_{\sigma,\lambda}(\overline{N}_j)=-\partial_j$ by \eqref{eq:NpictureActionNbar}, so $V$ consists of polynomials. Finally, by Proposition~\ref{prop:CompositionSeries} there is at most one non-trivial finite-dimensional subrepresentation of $I_{\sigma,\lambda}$.
\end{proof}

In view of the previous statement, we denote by $\widehat{I}_{\sigma,\lambda}\subseteq\mathcal{D}'(\mathbb{R}^n\setminus\{0\})\otimes V_\sigma$ the image of $\mathcal{F}(I_{\sigma,\lambda})$ under the restriction map $\mathcal{D}'(\mathbb{R}^n)\otimes V_\sigma\to\mathcal{D}'(\mathbb{R}^n\setminus\{0\})\otimes V_\sigma$. By \eqref{eq:FpictureActionNbar}, \eqref{eq:FpictureActionM}, \eqref{eq:FpictureActionA} and Lemma~\ref{lem:FpictureActionN}, the Lie algebra action $d\widehat{\pi}_{\sigma,\lambda}$ is by differential operators with polynomial coefficients, so it restricts to a Lie algebra action on $\widehat{I}_{\sigma,\lambda}$ which we denote by the same expression.

Now let $f\in\widehat{I}_{\sigma,\lambda}$. If we denote the stabilizer in $M$ of a point $\xi \in \mathbb{R}^n\setminus\{0\}$ by $M_\xi \cong \mathrm{SO}(n-1)$, then the restriction of $V_\sigma$ to $M_\xi$ decomposes into a multiplicity-free sum of irreducible representations. We would like to decompose $f(\xi)$ according to this decomposition. Since $M_\xi$ and therefore also the decomposition of $V_\sigma$ depends on $\xi$, we fix $e_1\in\mathbb{R}^n$ as a basepoint and decompose
\begin{equation}
	V_\sigma = \bigoplus_{\substack{\tau\in\widehat{M}_{e_1}\\\tau\preceq\sigma}} W_\tau.\label{eq:DecompositionSigmaIntoTau}
\end{equation}
For each $\xi \in \mathbb{R}^n\setminus\{0\}$ there exists some $m_\xi \in  M$, unique up to right multiplication by $M_{e_1}$ such that $  \lVert \xi  \rVert m_\xi e_1 = \xi $, so that $M_\xi = m_\xi M_{e_1 }m_{\xi }^{-1} $. Then the decomposition of $\sigma|_{M_\xi}$ into irreducibles can be written as
\begin{equation}
	V_\sigma = \bigoplus_{\substack{\tau\in\widehat{M}_{e_1}\\\tau\preceq\sigma}} W_\tau(\xi), \qquad \mbox{with }W_\tau(\xi)=\sigma(m_\xi)W_\tau.\label{eq:DecompositionSigmaMxi}
\end{equation}
Accordingly, we decompose
\begin{equation}
	\widehat{I}_{\sigma,\lambda} = \bigoplus_{\substack{\tau\in\widehat{M}_{e_1}\\\tau\preceq\sigma}} \widehat{I}_{\sigma,\lambda}(\tau),\label{eq:DecompositionIhat}
\end{equation}
where $\widehat{I}_{\sigma,\lambda}(\tau)$ consists of all $f\in\widehat{I}_{\sigma,\lambda}$ such that
\begin{align*}
\forall \xi \in \mathbb{R}^n\setminus\{0\} \colon \quad f (\xi ) \in W_\tau(\xi).
\end{align*}

Note that each subspace $\widehat{I}_{\sigma,\lambda}(\tau)$ is $\overline{P}$-invariant. This is clear for the action of $\overline{N}$ and $A$, and for each $m_0 \in M$:
\begin{align*}
\widehat{\pi}_{\sigma,\lambda} (m_0 )f (\xi )=\sigma (m_0 )f (m_0^{-1} \xi ) \in \sigma (m_0 ) \sigma (m_0^{-1} )\sigma (m_\xi )W_\tau = \sigma (m_\xi )W_\tau.
\end{align*}
In particular, we obtain:

\begin{lemma}\label{lemma:cnfpw3m9l9}
  For each $\tau \in \widehat{M}_{e_1 }$ the following operators leave the space $\widehat{I}_{\sigma,\lambda}(\tau)$ invariant:
  \begin{align*}
    d \widehat{\pi}_{\sigma, \lambda} (X_{ab})&=d \sigma (X_{ab})+\xi _a \partial_b - \xi _b \partial_a, \\
    \sum_{a,d=1}^{n}\xi_d d \widehat{\pi}_{\sigma, \lambda} (X_{aj})d \widehat{\pi}_{\sigma, \lambda} (X_{ad})&=\sum_{a=1}^{n}d \sigma (X_{aj})d \sigma (e_a \xi^t - \xi e_a^t )+d \sigma (\xi e_j^t - e_j \xi^t )E-\lVert \xi  \rVert^2 d \sigma (\partial e_j^t - e_j \partial^t )\\
                                                    & \qquad -\xi_j d \sigma (\partial  \xi^t - \xi \partial^t ) + \sum_{a,d=1}^{n}\xi_d (\xi_a \partial_j - \xi_j \partial_a )(\xi_a \partial_d - \xi_d \partial_a ),\\
    \sum_{a,b,d=1}^{n}\xi_b \xi_d d \widehat{\pi}_{\sigma, \lambda} (X_{ab})d \widehat{\pi}_{\sigma, \lambda} (X_{ad}) &=\sum_{a=1}^{n}d \sigma (e_a \xi ^t - \xi  e_a^t )^2 -2\lVert \xi  \rVert^2 d \sigma (\partial  \xi ^t - \xi  \partial ^t )\\
    & \qquad + \sum_{a,b,d=1}^{n}\xi _b \xi _d (\xi _a \partial_b - \xi _b \partial_a )(\xi _a \partial_d -\xi _d \partial_a ),
\end{align*}
where $X_{ab} \coloneqq e_a e_b^t - e_b e_a^t=E_{a,b}-E_{b,a} \in \mathfrak{m}$ and where we write $d\sigma(\partial\xi^t-\xi\partial^t)=\sum_{i,j=1}^nd\sigma(X_{ij})\xi_j\partial_i$.
\end{lemma}

\begin{proof}
  The first line is a direct consequence of Equation \eqref{eq:cnfpvp1t7g}. The other ones follow from
  \begin{align*}
    d \widehat{\pi}_{\sigma, \lambda} (X_{ab})d \widehat{\pi}_{\sigma, \lambda} (X_{c d})&=d \sigma (X_{ab})d \sigma (X_{c d})+d \sigma (X_{ab})(\xi _c \partial_d -\xi _d \partial_c )+(\xi _a \partial_b - \xi _b \partial_a )d \sigma (X_{c d})\\
    &\qquad +(\xi _a \partial_b - \xi_b \partial_a )(\xi _c \partial_d - \xi _d \partial_c )
\end{align*}
so that $\sum_{a,d=1}^{n}\xi_d d \widehat{\pi}_{\sigma, \lambda} (X_{aj})d \widehat{\pi}_{\sigma, \lambda} (X_{ad})$ equals
\begin{align*}
  &\hphantom{{}={}}\sum_{a=1}^{n}d \sigma  (X_{aj})d \sigma (e_a \xi^t - \xi e_a^t )+\sum_{a=1}^{n}d \sigma (X_{aj}) (\xi_a E - \lVert \xi  \rVert^2 \partial_a )\\
                                                    & \qquad +\sum_{a,d=1}^{n} \xi_d (\xi_a \partial_j - \xi_j \partial_a )d \sigma (X_{ad})+\sum_{a, d=1}^{n}\xi_d (\xi_a \partial_j - \xi_j \partial_a )(\xi_a \partial_d - \xi_d \partial_a )\\
                                                    &=\sum_{a=1}^{n}d \sigma (X_{aj})d \sigma (e_a \xi^t - \xi e_a^t )+d \sigma (\xi e_j^t - e_j \xi^t )E-\lVert \xi  \rVert^2 d \sigma (\partial e_j^t - e_j \partial^t )\\
                                                    & \qquad +\sum_{a,d=1}^{n}((\xi_a \partial_j - \xi_j \partial_a )\xi_d + \xi_j \delta_{a d}- \xi_a \delta_{jd})d \sigma (X_{ad})+\sum_{a,d=1}^{n}\xi_d (\xi_a \partial_j - \xi_j \partial_a )(\xi_a \partial_d - \xi_d \partial_a )\\
                                                    &=\sum_{a=1}^{n}d \sigma (X_{aj})d \sigma (e_a \xi^t - \xi e_a^t )+d \sigma (\xi e_j^t - e_j \xi^t )E-\lVert \xi  \rVert^2 d \sigma (\partial e_j^t - e_j \partial^t )-d \sigma (\xi e_j^t - e_j \xi^t )\\
  & \qquad +\sum_{a=1}^{n}(\xi_a \partial_j - \xi_j \partial_a ) d \sigma (e_a \xi^t - \xi e_a^t )+\sum_{a,d=1}^{n}\xi_d (\xi_a \partial_j - \xi_j \partial_a )(\xi_a \partial_d - \xi_d \partial_a )\\
                                                    &=\sum_{a=1}^{n}d \sigma (X_{aj})d \sigma (e_a \xi^t - \xi e_a^t )+d \sigma (\xi e_j^t - e_j \xi^t )E-\lVert \xi  \rVert^2 d \sigma (\partial e_j^t - e_j \partial^t )-d \sigma (\xi e_j^t - e_j \xi^t )\\
                                                    & \qquad -\xi_j d \sigma (\partial  \xi^t - \xi \partial^t ) + d \sigma (\xi e_j^t - e_j \xi^t )+\sum_{a,d=1}^{n}\xi_d (\xi_a \partial_j - \xi_j \partial_a )(\xi_a \partial_d - \xi_d \partial_a ).\qedhere
\end{align*}
\end{proof}

We now investigate some Casimir operators. For this we first define the normalized Killing form
\begin{align*}
\tilde{B}(X, Y)\coloneqq - \frac{1}{2} \mathrm{tr}(XY)
\end{align*}
on $\mathfrak{g}$. Note that with respect to this form each orthonormal basis $v_1 , \ldots , v_{n-1} \in \mathbb{R}^n $ of $e_1^\perp\subseteq\mathbb{R}^n$ gives rise to an orthonormal basis $e_1  v_j^t - v_j e_1 ^t $ of $\mathfrak{m}_{e_1  }^\perp $. Indeed, $(e_1  v_j^t - v_j e_1 ^t )(e_1  v_k^t - v_k e_1 ^t )=-\delta_{jk}e_1  e_1 ^t - v_j v_k^t $ so that $\tilde{B}(e_1 v_j^t - v_j e_1^t , e_1 v_k^t - v_k e_1^t )=\delta_{jk}$. Moreover, $\tilde{B}(e_1 v_j^t - v_j e_1^t, X) = 0$ for each $X \in \mathfrak{m}_{e_1 }=1 \times \mathfrak{so}(n-1)$. Similarly, we obtain an orthonormal basis at any point $\xi \in \mathbb{R}^n $ by conjugating with $m_\xi $, i.e., $\frac{1}{\lVert \xi  \rVert}(\xi \eta^t - \eta \xi^t)$ is an orthonormal basis of $\mathfrak{m}_{\xi }^\perp $ whenever $\eta $ runs through an orthonormal basis of $\xi^\perp$. In particular, the Casimir operator of $\mathfrak{m}_\xi^\perp $ with respect to $\sigma $ is given by
\begin{align*}
\mathrm{Cas}(\mathfrak{m}_\xi^\perp )=\frac{1}{\lVert \xi \rVert^2 }\sum_{i=1}^{n}d \sigma (e_i \xi^t - \xi e_i^t )^2 ,
\end{align*}
where we used that we may let $\eta $ from above run through an arbitrary basis of $\mathbb{R}^n $ since $\xi \xi^t - \xi \xi^t =0$. In the following we will show that
\begin{align*}
[\partial_j , \mathrm{Cas}(\mathfrak{m}_{\xi }^\perp )]+\mathcal{B}_{0,j}^\sigma \colon \quad  \widehat{I}_{\sigma,\lambda}(\tau) \to \widehat{I}_{\sigma,\lambda}(\tau).
\end{align*}
We first calculate the commutator.
\begin{lemma}
  We have
  \begin{align*}
[\partial_j , \mathrm{Cas}(\mathfrak{m}_{\xi }^\perp )]=\frac{-2 \xi_j }{\lVert \xi  \rVert^4 }\sum_{a=1}^{n}d \sigma (e_a \xi^t - \xi e_a^t )^2 +\frac{n-2}{\lVert \xi  \rVert^2 } \sum_{\ell =1}^{n}\xi_{\ell} d \sigma (X_{j \ell })+\frac{2}{\lVert \xi  \rVert^2 } \sum_{a=1}^{n}d \sigma (X_{a j})d \sigma (e_a \xi^t - \xi e_a^t ).
\end{align*}
\end{lemma}
\begin{proof}
  We directly obtain that
  \begin{align*}
    [\partial_j , \mathrm{Cas}(\mathfrak{m}_{\xi }^\perp )]=\frac{-2 \xi_j }{\lVert \xi  \rVert^4 }\sum_{a=1}^{n}d \sigma (e_a \xi^t - \xi e_a^t )^2 +\frac{1}{\lVert \xi  \rVert^2 }\sum_{a=1}^{n}\Big(d \sigma (X_{aj})d \sigma (e_a \xi^t - \xi e_a^t )+d \sigma (e_a \xi^t - \xi e_a^t )d \sigma (X_{aj})\Big).
\end{align*}
Now note that $[ e_a \xi^t - \xi e_a^t , X_{aj}] = \sum_{\ell =1}^{n}\xi_{\ell} [X_{a \ell },X_{aj}]=\sum_{\ell =1}^{n}\xi_{\ell} ( \delta_{a \ell }X_{aj}+\delta_{aj}X_{\ell a}+X_{j \ell })$ so that
\begin{align*}
  &\hphantom{{}={}}\sum_{a=1}^{n}\Big(d \sigma (X_{aj})d \sigma (e_a \xi^t - \xi e_a^t )+d \sigma (e_a \xi^t - \xi e_a^t )d \sigma (X_{aj})\Big)\\
  &= \sum_{a=1}^{n}\Big(2d \sigma (X_{aj})d \sigma (e_a \xi^t - \xi e_a^t )+\xi_a d \sigma (X_{a j})+\delta_{aj}\sum_{\ell =1}^{n}\xi_{\ell} d \sigma (X_{\ell a})+\sum_{\ell =1}^{n}\xi_{\ell}d \sigma (X_{j \ell })\Big)\\
  &= 2\sum_{a=1}^{n}d \sigma (X_{a j})d \sigma (e_a \xi^t - \xi e_a^t )+\sum_{a=1}^{n}\xi_a d \sigma (X_{a j})+\sum_{\ell =1}^{n}\xi_{\ell} d \sigma (X_{\ell j})+n \sum_{\ell =1}^{n} \xi_{\ell} d \sigma (X_{j \ell })\\
  &=2\sum_{a=1}^{n}d \sigma (X_{a j})d \sigma (e_a \xi^t - \xi e_a^t )+(n-2)\sum_{\ell =1}^{n}\xi_{\ell} d \sigma (X_{j \ell }).\qedhere
\end{align*}
\end{proof}
Before we use Lemma \ref{lemma:cnfpw3m9l9} to rewrite this commutator we first calculate some differential operators occurring in that lemma. By straightforward calculations we obtain the following:
\begin{lemma}\label{lemma:cnfqcxk32j}
  We have
  \begin{align*}
    \sum_{a,d=1}^{n}\xi _d (\xi _a \partial_j -\xi _j \partial_a )(\xi _a \partial_d - \xi _d \partial_a )&=\xi_j \lVert \xi  \rVert^2 \Delta - \xi _j E^2-(n-2)\xi_j E,\\
    \sum_{a,b,d=1}^{n}\xi _b \xi _d (\xi _a \partial_b - \xi _b \partial_a )(\xi _a \partial_d -\xi _d \partial_a )&=\lVert \xi  \rVert^4 \Delta - \lVert \xi \rVert^2 E^2 - (n-2)\lVert \xi  \rVert^2 E.
\end{align*}
\end{lemma}

We can now prove the aforementioned invariance result.
\begin{proposition}\label{proposition:cnfqtl00hg}
The operator $[\partial_j , \mathrm{Cas}(\mathfrak{m}_\xi^\perp )]+\mathcal{B}_{0,j}^\sigma $ leaves each $\widehat{I}_{\sigma,\lambda}(\tau)$ invariant.
\end{proposition}

\begin{proof}
  Using Lemma \ref{lemma:cnfpw3m9l9} and \ref{lemma:cnfqcxk32j} we find that $\frac{-2 \xi_j }{\lVert \xi  \rVert^4 }\sum_{a=1}^{n}d \sigma (e_a \xi^t - \xi e_a^t )^2 +\frac{2}{\lVert \xi  \rVert^2 } \sum_{a=1}^{n}d \sigma (X_{a j})d \sigma (e_a \xi^t - \xi e_a^t )$ equals
  \begin{align*}
    &\hphantom{{}={}}\frac{2}{\lVert \xi  \rVert^2 }\biggl(-\frac{\xi_j }{\lVert \xi  \rVert^2 }\sum_{a,b,d=1}^{n}\xi_b \xi_d d \widehat{\pi}_{\sigma, \lambda} (X_{ab})d \widehat{\pi}_{\sigma, \lambda} (X_{ad})-2 \xi_j d \sigma (\partial \xi^t - \xi \partial^t )+\xi_j (\lVert \xi  \rVert^2 \Delta - E^2 - (n-2)E)\\
    &\qquad\qquad +\sum_{a,d=1}^{n}\xi_d d \widehat{\pi}_{\sigma, \lambda} (X_{aj})d \widehat{\pi}_{\sigma, \lambda} (X_{ad})-d \sigma (\xi e_j^t -e_j \xi^t )E+\lVert \xi  \rVert^2 d \sigma (\partial e_j^t - e_j \partial^t )+ \xi_j d \sigma (\partial \xi^t - \xi \partial^t )\\
    &\qquad\qquad -\Bigl(\xi_j \lVert \xi  \rVert^2 \Delta - \xi_j E^2 - (n-2)\xi_j E\Bigr)\biggr)\\
    &=\frac{2}{\lVert \xi  \rVert^2 }\biggl(-\frac{\xi_j }{\lVert \xi  \rVert^2 }\sum_{a,b,d=1}^{n}\xi_b \xi_d d \widehat{\pi}_{\sigma, \lambda} (X_{ab})d \widehat{\pi}_{\sigma, \lambda} (X_{ad})+\sum_{a,d=1}^{n}\xi_d d \widehat{\pi}_{\sigma, \lambda} (X_{aj})d \widehat{\pi}_{\sigma, \lambda} (X_{ad})\\
    &\qquad\qquad -\xi_j d \sigma (\partial \xi^t - \xi \partial^t )-d \sigma (\xi e_j^t -e_j \xi^t )E+\lVert \xi  \rVert^2 d \sigma (\partial e_j^t - e_j \partial^t )\biggr).
  \end{align*}
  Thus, using Lemma \ref{lemma:cnfpw3m9l9} again, we see that $[\partial_j , \mathrm{Cas}(\mathfrak{m}_{\xi }^\perp )]+\mathcal{B}_{0,j}^\sigma $ is given by
  \begin{align*}
    &\hphantom{{}={}}\frac{2}{\lVert \xi  \rVert^2 }\biggl(-\frac{\xi_j }{\lVert \xi  \rVert^2 }\sum_{a,b,d=1}^{n}\xi_b \xi_d d \widehat{\pi}_{\sigma, \lambda} (X_{ab})d \widehat{\pi}_{\sigma, \lambda} (X_{ad})+\sum_{a,d=1}^{n}\xi_d d \widehat{\pi}_{\sigma, \lambda} (X_{aj})d \widehat{\pi}_{\sigma, \lambda} (X_{ad})\\
    &\qquad\qquad -\xi_j d \sigma (\partial \xi^t - \xi \partial^t )-d \sigma (\xi e_j^t -e_j \xi^t )E+\lVert \xi  \rVert^2 d \sigma (\partial e_j^t - e_j \partial^t )\biggr)\\
    &\qquad +\frac{n-2}{\lVert \xi  \rVert^2 }\Bigl(d \widehat{\pi}_{\sigma, \lambda} (e_j \xi^t - \xi e_j^t )- \xi_j E+ \lVert \xi  \rVert^2 \partial_j \Bigr)+ \xi_j \Delta -2 (E + \rho )\partial_j - 2 d \sigma (\partial e_j^t - e_j \partial^t )\\
    &=\frac{2}{\lVert \xi  \rVert^2 }\biggl(-\frac{\xi_j }{\lVert \xi  \rVert^2 }\sum_{a,b,d=1}^{n}\xi_b \xi_d d \widehat{\pi}_{\sigma, \lambda} (X_{ab})d \widehat{\pi}_{\sigma, \lambda} (X_{ad})+\sum_{a,d=1}^{n}\xi_d d \widehat{\pi}_{\sigma, \lambda} (X_{aj})d \widehat{\pi}_{\sigma, \lambda} (X_{ad})-\frac{n-2}{2}\xi_j E\biggr)\\
    &\qquad -\frac{2\xi_j}{\lVert \xi  \rVert^2 } d \sigma (\partial \xi^t - \xi \partial^t )-\frac{2}{\lVert \xi  \rVert^2 }d \sigma (\xi e_j^t -e_j \xi^t )E+\frac{n-2}{\lVert \xi  \rVert^2 }d \widehat{\pi}_{\sigma, \lambda} (e_j \xi^t - \xi e_j^t )+ \xi_j \Delta -2 E\partial_j-2\partial_j \\
    &=\frac{2}{\lVert \xi  \rVert^2 }\biggl(-\frac{\xi_j }{\lVert \xi  \rVert^2 }\sum_{a,b,d=1}^{n}\xi_b \xi_d d \widehat{\pi}_{\sigma, \lambda} (X_{ab})d \widehat{\pi}_{\sigma, \lambda} (X_{ad})+\sum_{a,d=1}^{n}\xi_d d \widehat{\pi}_{\sigma, \lambda} (X_{aj})d \widehat{\pi}_{\sigma, \lambda} (X_{ad})-\frac{n-2}{2}\xi_j E\biggr)\\
    &\qquad -\frac{2\xi_j}{\lVert \xi  \rVert^2 } d \sigma (\partial \xi^t - \xi \partial^t )-\frac{2}{\lVert \xi  \rVert^2 }\Bigl(d \widehat{\pi}_{\sigma, \lambda} (\xi e_j^t -e_j \xi^t )+\xi_j E- \lVert \xi  \rVert^2 \partial_j \Bigr)E+\frac{n-2}{\lVert \xi  \rVert^2 }d \widehat{\pi}_{\sigma, \lambda} (e_j \xi^t - \xi e_j^t )\\
    &\qquad +\xi_j \Delta -2 E\partial_j-2\partial_j \\
    &=\frac{2}{\lVert \xi  \rVert^2 }\biggl(-\frac{\xi_j }{\lVert \xi  \rVert^2 }\sum_{a,b,d=1}^{n}\xi_b \xi_d d \widehat{\pi}_{\sigma, \lambda} (X_{ab})d \widehat{\pi}_{\sigma, \lambda} (X_{ad})+\sum_{a,d=1}^{n}\xi_d d \widehat{\pi}_{\sigma, \lambda} (X_{aj})d \widehat{\pi}_{\sigma, \lambda} (X_{ad})-\frac{n-2}{2}\xi_j E\biggr)\\
    &\qquad -\frac{2\xi_j}{\lVert \xi  \rVert^2 } d \sigma (\partial \xi^t - \xi \partial^t )-\frac{2\xi_j}{\lVert \xi  \rVert^2 } E^2+\frac{n}{\lVert \xi  \rVert^2 }d \widehat{\pi}_{\sigma, \lambda} (e_j \xi^t - \xi e_j^t )+ \xi_j \Delta \\
    &=\frac{2}{\lVert \xi  \rVert^2 }\biggl(-\frac{\xi_j }{\lVert \xi  \rVert^2 }\sum_{a,b,d=1}^{n}\xi_b \xi_d d \widehat{\pi}_{\sigma, \lambda} (X_{ab})d \widehat{\pi}_{\sigma, \lambda} (X_{ad})+\sum_{a,d=1}^{n}\xi_d d \widehat{\pi}_{\sigma, \lambda} (X_{aj})d \widehat{\pi}_{\sigma, \lambda} (X_{ad})+\frac{n}{2}d \widehat{\pi}_{\sigma, \lambda} (e_j \xi^t - \xi e_j^t )\biggr)\\
    &\qquad +\frac{\xi_j}{\lVert \xi  \rVert^2 } \sum_{\ell =1}^{n}\xi_{\ell} \mathcal{B}_{0,\ell }^\sigma.
  \end{align*}
  The proposition now follows from the next lemma.
\end{proof}

\begin{lemma}
The operator $\sum_{\ell =1}^{n}\xi_{\ell} \mathcal{B}_{0,\ell }^\sigma $ leaves each $\widehat{I}_{\sigma,\lambda}(\tau)$ invariant.
\end{lemma}

\begin{proof}
  Let $X_1 , \ldots, X_{\dim(\mathfrak{m})}$ be any orthonormal basis of $\mathfrak{m}$ with respect to $\tilde{B}$. Then
  \begin{align*}
  \{H,X_1 , \ldots , X_{\dim (\mathfrak{m})}, N_1 , \ldots , N_n , \overline{N}_1 ,\ldots ,  \overline{N}_n\}
\end{align*}
  is a basis of $\mathfrak{g}$ with dual basis $-H,X_1 , \ldots , X_{\dim (\mathfrak{m})}, \overline{N}_1 , \ldots , \overline{N}_n , N_1,\ldots , N_n $.
  Thus, the Casimir element of $\mathfrak{g}$ is given by
  \begin{align*}
\mathrm{Cas}(\mathfrak{g})=-H^2 + \sum_{i=1}^{n}X_i^2 +\sum_{j=1}^{n}\overline{N}_j N_j +N_j \overline{N}_j =-H^2 + \sum_{i=1}^{n}X_i^2 + \sum_{j=1}^{n}2\overline{N}_j N_j+[N_j , \overline{N}_j ],
  \end{align*}
  where all terms $H, X_i^2$ and $[N_j , \overline{N}_j ]$ belong to the universal enveloping algebra of $\mathfrak{m} \oplus \mathfrak{a}$. Since $\mathfrak{a},\mathfrak{m}$ and $\mathrm{Cas}(\mathfrak{g})$ leave $\widehat{I}_{\sigma,\lambda}(\tau)$ invariant, so does the action of $\sum_{j=1}^{n}\overline{N}_j N_j $. But
    \begin{gather*}
\sum_{j=1}^{n}d \widehat{\pi}_{\sigma , \lambda } (\overline{N}_j) d \widehat{\pi}_{\sigma , \lambda } (N_j )=-\sum_{j=1}^{n}\xi_j \mathcal{B}_{\lambda , j}^\sigma=-\sum_{j=1}^{n}\xi_j \mathcal{B}_{0, j}^\sigma - 2 \lambda E,
\end{gather*}
where the Euler operator leaves $\widehat{I}_{\sigma,\lambda}(\tau)$ invariant since it comes from the $\mathfrak{a}$-action.
\end{proof}

This proposition allows us to prove the main result of this section.

\begin{proposition}\label{proposition:cnfql9m3l9}
Let $\tau_1 , \tau_2 \in \widehat{M}_{e_1 }$ be two inequivalent representations. Then
  \begin{align*}
i\,\mathrm{pr}_{\widehat{I}_{\sigma,\lambda}(\tau_2)}(d \widehat{\pi}_{\sigma , \lambda }(N_j )|_{\widehat{I}_{\sigma,\lambda}(\tau_1)})=\mathrm{pr}_{\widehat{I}_{\sigma,\lambda}(\tau_2)}(([\mathrm{Cas}(\mathfrak{m}_\xi^\perp ), \partial_j ] + 2 \lambda \partial_j)|_{\widehat{I}_{\sigma,\lambda}(\tau_1)}),
\end{align*}
where $\mathrm{pr}_{\widehat{I}_{\sigma,\lambda}(\tau_2)}$ denotes the projection onto $\widehat{I}_{\sigma,\lambda}(\tau_2)$ in the decomposition \eqref{eq:DecompositionIhat}.
\end{proposition}

\begin{proof}
  We write
  \begin{align*}
i\,d \widehat{\pi}_{\sigma , \lambda }(N_j )&=\mathcal{B}_{0 , j}^\sigma +2 \lambda \partial_j = \mathcal{B}_{0,j}^\sigma + [\partial_j , \mathrm{Cas}(\mathfrak{m}_\xi^\perp )]-[\partial_j , \mathrm{Cas}(\mathfrak{m}_\xi^\perp )]+2 \lambda \partial_j
  \end{align*}
  and apply Proposition \ref{proposition:cnfqtl00hg}.
\end{proof}

\section{Eigenvalues of intertwining operators and explicit Hilbert spaces}

In this section we find an explicit expression for the standard intertwining operators $A_{\sigma,\lambda}$ in the F-picture and use it to describe the invariant inner products on irreducible unitarizable subrepresentations and quotients. For this, we first find a description of all irreducible subrepresentations and quotients in the F-picture.

\subsection{Subrepresentations in the F-picture}

Recall the decomposition \eqref{eq:DecompositionIhat} of $\widehat{I}_{\sigma,\lambda}$ into the $\overline{P}$-invariant subspaces $\widehat{I}_{\sigma,\lambda}(\tau)$, where $\tau\in\widehat{M}_{e_1}$, $\tau\preceq\sigma$. In order to identify subrepresentations, it suffices to understand how the $\mathfrak{n}$-action connects the different subspaces $\widehat{I}_{\sigma,\lambda}(\tau)$. To describe this action in more detail, we use the same parametrization as in Section~\ref{sec:PrincipalSeriesAndIntertwiningOperators} for the unitary dual of $M_{e_1}\cong\mathrm{SO}(n-1)$. Abusing notation, the highest weight of an irreducible unitary representation $\tau$ of $M_{e_1}$ is a tuple $\tau\in\mathbb{Z}^{m-1}$ such that $\tau_1\geq\ldots\geq\tau_{m-2}\geq\tau_{m-1}\geq0$ if $n=2m$ is even and $\tau_1\geq\ldots\geq\tau_{m-2}\geq|\tau_{m-1}|$ if $n=2m-1$ is odd. Then $\tau\preceq\sigma$ if and only if
\begin{alignat*}{2}
&\sigma_1\geq \tau_1\geq \sigma_2\geq \tau_2\geq \ldots \geq \tau_{m-1}\geq |\sigma_m| && \text{if }n=2m,\\
&\sigma_1\geq \tau_1\geq \sigma_2\geq \tau_2\geq \ldots \geq \tau_{m-2}\geq \sigma_{m-1}\geq |\tau_{m-1}|\qquad && \text{if }n=2m-1.
\end{alignat*}

By Proposition \ref{proposition:cnfql9m3l9}, the $\mathfrak{n}$-action can only reach those $\widehat{I}_{\sigma,\lambda}(\tau')$ that are reachable by one application of $\partial_j$ for some $j\in\{1,\ldots, n\}$. The following lemma describes these possibilities.

\begin{lemma}\label{lemma:cnfql9m3l9}
Let $f\in\widehat{I}_{\sigma,\lambda}(\tau)$ and $v\in\mathbb{R}^n$, then
$$ \partial_vf \in \widehat{I}_{\sigma,\lambda}(\tau)\oplus\bigoplus_{i=1}^{m-1}\Big(\widehat{I}_{\sigma,\lambda}(\tau+e_i)\oplus \widehat{I}_{\sigma,\lambda}(\tau-e_i)\Big). $$
\end{lemma}

\begin{proof}
We have to show that
$$ \forall\xi\in\mathbb{R}^n\setminus\{0\}: \qquad \partial_vf(\xi)\in W_\tau(\xi)\oplus\bigoplus_{i=1}^{m-1}\Big(W_{\tau+e_i}(\xi)\oplus W_{\tau-e_i}(\xi)\Big). $$
Fix $\xi\in\mathbb{R}^n\setminus\{0\}$ we write $v=\alpha\xi+w$ with $\alpha\in\mathbb{R}$ and $w\perp\xi$. Since $\mathfrak{m}\xi=\xi^\perp$, there exists $T\in\mathfrak{m}$ such that $T\xi=w$. Together we find
\begin{align*}
	\partial_vf(\xi) &= \alpha\partial_\xi f(\xi)+\partial_{T\xi}f(\xi) = \alpha Ef(\xi) + \partial_{T\xi}f(\xi)\\
	&= -d\widehat{\pi}_{\sigma,\lambda}(H_0)f(\xi)+(\lambda-\rho)f(\xi)-d\widehat{\pi}_{\sigma,\lambda}(T)f(\xi)+d\sigma(T)f(\xi).
\end{align*}
The second term is clearly contained in $W_\tau(\xi)$. Moreover, since both $M$ and $A$ leave $\widehat{I}_{\sigma,\lambda}(\tau)$ invariant, both $d\widehat{\pi}_{\sigma,\lambda}(H_0)f$ and $d\widehat{\pi}_{\sigma,\lambda}(T)f$ are contained in $\widehat{I}_{\sigma,\lambda}(\tau)$, so the first and the third terms are also contained in $W_\tau(\xi)$. Finally, the fourth term is contained in
$$ d\sigma(T)W_\tau(\xi) = d\sigma(T)\sigma(m_\xi)W_\tau = \sigma(m_\xi)d\sigma(\operatorname{Ad}(m_\xi)^{-1}T)W_\tau \subseteq \sigma(m_\xi)d\sigma(\mathfrak{m})W_\tau. $$
Clearly $d\sigma(\mathfrak{m}_{e_1})W_\tau\subseteq W_\tau$. To decompose $d\sigma(\mathfrak{m}_{e_1}^\perp)W_\tau$, we first note that, as a representation of $M_{e_1}$, we have $\mathfrak{m}_{e_1}^\perp\simeq\mathbb{R}^{n-1}$. Hence, $d\sigma(\mathfrak{m}_{e_1}^\perp)W_\tau$ is a quotient of $W_{(1,0,\ldots,0)}\otimes W_\tau$. The possible highest weights in the tensor product are sums of $\tau$ and weights of $W_{(1,0,\ldots,0)}$. The latter are given by $\pm e_i$ and possibly $0$. This shows that
$$ d\sigma(\mathfrak{m})W_\tau \subseteq W_\tau\oplus\bigoplus_{i=1}^{m-1}\Big(W_{\tau+e_i}\oplus W_{\tau-e_i}\Big) $$
and the proof is complete.
\end{proof}

Whether $\widehat{I}_{\sigma,\lambda}(\tau\pm e_i)$ can in fact be reached from $\widehat{I}_{\sigma,\lambda}(\tau)$ by the $\mathfrak{n}$-action can now be determined using Proposition~\ref{proposition:cnfql9m3l9}:

\begin{proposition}\label{prop:NactionInFpicture}
Let $\tau\in\widehat{M}_{e_1}$ and $f\in\widehat{I}_{\sigma,\lambda}(\tau)$. Then for $1\leq j\leq n$:
$$ d\widehat{\pi}_{\sigma,\lambda}(N_j)f \in \widehat{I}_{\sigma,\lambda}(\tau)\oplus\bigoplus_{i=1}^{m-1}\Big(\widehat{I}_{\sigma,\lambda}(\tau+e_i)\oplus \widehat{I}_{\sigma,\lambda}(\tau-e_i)\Big) $$
and
\begin{equation}
	\mathrm{pr}_{\widehat{I}_{\sigma,\lambda}(\tau\pm e_i)}\big(d\widehat{\pi}_{\sigma,\lambda}(N_j)f\big) = -i\big(2\lambda+1\pm2(\tau_i+\tfrac{n-2i-1}{2})\big)\,\mathrm{pr}_{\widehat{I}_{\sigma,\lambda}(\tau\pm e_i)}\big(\partial_jf\big) \qquad (1\leq i\leq m-1).\label{eq:NactionInFpicture}
\end{equation}
\end{proposition}

\begin{proof}
The first claim follows immediately from Proposition~\ref{proposition:cnfql9m3l9}, Lemma \ref{lemma:cnfql9m3l9} and the fact that $\mathrm{Cas}(\mathfrak{m}_\xi^\perp)=\mathrm{Cas}(\mathfrak{m})-\mathrm{Cas}(\mathfrak{m}_\xi)$ acts by a scalar on each $W_\tau(\xi)$. To show the second claim we note that $\mathrm{Cas}(\mathfrak{m})$ acts on both $W_{\tau}$ and $W_{\tau\pm e_i}$ by the same scalar (depending only on $\sigma$). Thus,
\begin{multline*}
  \mathrm{pr}_{\widehat{I}_{\sigma,\lambda}(\tau\pm e_i)}([\mathrm{Cas}(\mathfrak{m}_\xi^\perp),\partial_j]f)=\mathrm{pr}_{\widehat{I}_{\sigma,\lambda}(\tau\pm e_i)}([\mathrm{Cas}(\mathfrak{m}),\partial_j]f)-\mathrm{pr}_{\widehat{I}_{\sigma,\lambda}(\tau\pm e_i)}([\mathrm{Cas}(\mathfrak{m}_\xi),\partial_j]f)\\
  =\mathrm{pr}_{\widehat{I}_{\sigma,\lambda}(\tau\pm e_i)}([\partial_j,\mathrm{Cas}(\mathfrak{m}_\xi)]f).
\end{multline*}
Since the action of $\mathrm{Cas}(\mathfrak{m}_\xi)$ on $W_{\tau}(\xi)$ is given by
\begin{gather}\label{gather:cngwmab562}
	\mathrm{Cas}(\mathfrak{m}_\xi)|_{W_\tau(\xi)} = -\langle \tau+2\rho_{\mathfrak{so}(n-1)},\tau\rangle\cdot\mathrm{id}_{W_\tau(\xi)},
\end{gather}
where $\rho_{\mathfrak{so}(n-1)}=\frac{1}{2}(n-3,n-5,\ldots,n-1-2\lfloor\frac{n-1}{2}\rfloor)$ is the half sum of the positive roots of $\mathfrak{so}(n-1)$ and $\langle\cdot,\cdot\rangle$ the standard inner product on $\mathbb{R}^{m-1}$, we obtain
\begin{align*}
\mathrm{pr}_{\widehat{I}_{\sigma,\lambda}(\tau\pm e_i)}([\mathrm{Cas}(\mathfrak{m}_\xi^\perp),\partial_j]f)&=(\langle\tau\pm e_i+2\rho_{\mathfrak{so}(n-1)},\tau\pm e_i\rangle-\langle\tau+2\rho_{\mathfrak{so}(n-1)},\tau\rangle)\mathrm{pr}_{\widehat{I}_{\sigma,\lambda}(\tau\pm e_i)}(\partial_jf)\\
&=(1\pm 2\langle\tau+\rho_{\mathfrak{so}(n-1)},e_i\rangle)\mathrm{pr}_{\widehat{I}_{\sigma,\lambda}(\tau\pm e_i)}(\partial_jf)\\
&=(1\pm2(\tau_i+\tfrac{n-2i-1}{2}))\mathrm{pr}_{\widehat{I}_{\sigma,\lambda}(\tau\pm e_i)}(\partial_jf).
\end{align*}
Together with Proposition~\ref{proposition:cnfql9m3l9} this shows the claimed formula.
\end{proof}

We can now identify the irreducible subrepresentations and quotients of $\widehat{I}_{\sigma,\lambda}$:

\begin{proposition}\label{prop:SubrepsAndQuotientsInFpicture}
Let $\sigma\in\widehat{M}$, $1\leq i\leq m-1$ and $0\leq j<\sigma_i-|\sigma_{i+1}|$.
\begin{enumerate}[(i)]
	\item For $i\neq\frac{n-1}{2}$, the irreducible subrepresentation $\mathfrak{I}(\sigma,i,j)$ resp. irreducible quotient $\mathfrak{Q}(\sigma,i,j)$ of $I_{\sigma,\lambda}$ for $\lambda=\rho-i+|\sigma_{i+1}|+j$ is via the Fourier transform $\mathcal{F}:I_{\sigma,\lambda}\to\widehat{I}_{\sigma,\lambda}$ isomorphic to
	$$ \widehat{\mathfrak{I}}(\sigma,i,j)=\bigoplus_{\substack{\tau\preceq\sigma\\\tau_i>|\sigma_{i+1}|+j}}\widehat{I}_{\sigma,\lambda}(\tau) \qquad \mbox{resp.} \qquad \widehat{\mathfrak{Q}}(\sigma,i,j)=\widehat{I}_{\sigma,\lambda}/\widehat{\mathfrak{I}}(\sigma,i,j)\simeq\bigoplus_{\substack{\tau\preceq\sigma\\\tau_i\leq|\sigma_{i+1}|+j}}\widehat{I}_{\sigma,\lambda}(\tau). $$
	\item For $i=\frac{n-1}{2}$, the sum of the two irreducible subrepresentations $\mathfrak{I}(\sigma,i,j)^+\oplus\mathfrak{I}(\sigma,i,j)^-$ resp. the irreducible quotient $\mathfrak{Q}(\sigma,i,j)$ of $I_{\sigma,\lambda}$ for $\lambda=\rho-i+|\sigma_{i+1}|+j$ is via the Fourier transform $\mathcal{F}:I_{\sigma,\lambda}\to\widehat{I}_{\sigma,\lambda}$ isomorphic to $\widehat{\mathfrak{I}}(\sigma,i,j)^+\oplus\widehat{\mathfrak{I}}(\sigma,i,j)^-$ resp. $\widehat{\mathfrak{Q}}(\sigma,i,j)$, where
	$$ \widehat{\mathfrak{I}}(\sigma,i,j)^\pm=\bigoplus_{\substack{\tau\preceq\sigma\\\pm\tau_i>j}}\widehat{I}_{\sigma,\lambda}(\tau) \qquad \mbox{resp.} \qquad \widehat{\mathfrak{Q}}(\sigma,i,j)=\widehat{I}_{\sigma,\lambda}/\bigl(\widehat{\mathfrak{I}}(\sigma,i,j)^+\oplus\widehat{\mathfrak{I}}(\sigma,i,j)^-\bigr)\simeq\bigoplus_{\substack{\tau\preceq\sigma\\|\tau_i|\leq j}}\widehat{I}_{\sigma,\lambda}(\tau). $$
\end{enumerate}
\end{proposition}

\begin{proof}
It follows from Proposition~\ref{prop:NactionInFpicture} that
$$ \bigoplus_{\substack{\tau\preceq\sigma\\\tau_i>|\sigma_{i+1}|+j}}\widehat{I}_{\sigma,\lambda}(\tau) \qquad \mbox{resp.} \qquad \bigoplus_{\substack{\tau\preceq\sigma\\\pm\tau_i>j}}\widehat{I}_{\sigma,\lambda}(\tau) $$
is $\mathfrak{g}$-invariant and hence $G$-invariant. Moreover, every $\widehat{I}_{\sigma,\lambda}(\tau)$ is non-trivial: it contains at least all functions $f\in C_c^\infty(\mathbb{R}^n\setminus\{0\})\otimes V_\sigma$ such that $f(\xi)\in W_\tau(\xi)$ for all $\xi$ since their Fourier transform is contained in $\mathcal{S}(\mathbb{R}^n)\otimes V_\sigma\subseteq I_{\sigma,\lambda}$ (see \eqref{eq:FPictureSchwartz}). Since there is precisely one irreducible subrepresentation for $i\neq\frac{n-1}{2}$ and two for $i=\frac{n-1}{2}$ by Proposition~\ref{prop:CompositionSeries}, the claim follows.
\end{proof}

\begin{remark}
Since $\widehat{\mathfrak{I}}(\sigma,i,j)^+$ and $\widehat{\mathfrak{I}}(\sigma,i,j)^-$ are both $\mathfrak{g}$-invariant, it follows that the Fourier transform maps $\mathfrak{I}(\sigma,i,j)^+$ onto one of them, but we were not able to determine onto which of the two.
\end{remark}

\subsection{The F-picture of intertwining operators}

Recall that for $\sigma\in\widehat{M}$ with $\sigma\simeq w_0\sigma$ there exists a holomorphic family of intertwining operators
\begin{align*}
A_{\sigma , \lambda }\colon I_{\sigma,\lambda}\to I_{\sigma,-\lambda}.
\end{align*}
It follows from \cite[Chapter 8.3.1]{KobayashiSpeh18} that these operators are given by convolution
$$ A_{\sigma,\lambda}f(x) = (K_{\sigma,\lambda}*f)(x) = \int_{\mathbb{R}^n}K_{\sigma,\lambda}(x-y)f(y)\,dy \qquad (f\in I_{\sigma,\lambda}), $$
with a distribution kernel $K_{\sigma,\lambda}\in\mathcal{S}'(\mathbb{R}^n)\otimes\mathrm{End}(V_\sigma)$. This kernel can be written as
$$ K_{\sigma,\lambda}(x) = \mathrm{const}\times \lVert x \rVert^{2(\lambda-\rho)}\sigma\left(I_n-2\frac{xx^t}{\|x\|^2}\right) \qquad (x\in\mathbb{R}^n\setminus\{0\}), $$
where $\sigma$ is extended to a representation of $\mathrm{O}(n)$ (unique up to twisting by the determinant character) and where the constant depends on $\sigma$ and $\lambda$ and is chosen such that $K_{\sigma,\lambda}$ depends holomorphically on $\lambda\in\mathbb{C}$ and is nowhere vanishing.

Taking the Fourier transform, we obtain intertwining operators $\mathcal{F}(I_{\sigma,\lambda})\to\mathcal{F}(I_{\sigma,-\lambda}),\,f\mapsto\widehat{K}_{\sigma,\lambda}\cdot f$ given by multiplication with $\widehat{K}_{\sigma,\lambda}\in \mathcal{S}'(\mathbb{R}^n)\otimes \mathrm{End}(V_\sigma)$. By Proposition~\ref{prop:CompositionSeries} and \eqref{eq:NormalizedEigenvaluesIntertwiners}, every finite-dimensional subrepresentation of $\mathcal{F}(I_{\sigma,\lambda})$ is contained in the kernel of $\widehat{A}_{\sigma,\lambda}$, so by Lemma~\ref{lem:KernelRestriction} the intertwining operator descends to
$$ \widehat{A}_{\sigma,\lambda}:\widehat{I}_{\sigma,\lambda} \to \widehat{I}_{\sigma,-\lambda}, \quad \widehat{A}_{\sigma,\lambda}f = \widehat{K}_{\sigma,\lambda}|_{\mathbb{R}^n\setminus\{0\}}\cdot f. $$
The goal of this section is to find an explicit expression for the multiplier $\widehat{K}_{\sigma,\lambda}|_{\mathbb{R}^n\setminus\{0\}}$.

By the $A$-equivariance of $\widehat{A}_{\sigma, \lambda}$ we find that $\widehat{K}_{\sigma,\lambda}(e^{-t}\xi)=e^{2\lambda t}\widehat{K}_{\sigma,\lambda}(\xi)$ and hence
\begin{align*}
\widehat{K}_{\sigma,\lambda}(\xi)=\lVert \xi  \rVert^{-2\lambda}\widehat{K}_{\sigma,\lambda}\left(\frac{\xi}{\lVert\xi\rVert}\right)\eqqcolon \lVert\xi\rVert^{-2\lambda}\widetilde{K}_{\sigma,\lambda}(\xi).
\end{align*}
Moreover, the $M$-equivariance implies
\begin{align}\label{align:cnfsn4exfz}
\sigma(m^{-1})\circ\widetilde{K}_{\sigma,\lambda}(\xi)\circ\sigma(m)=\widetilde{K}_{\sigma,\lambda}(m^{-1}\xi),
\end{align}
in particular, $\widetilde{K}_{\sigma,\lambda}(\xi)$ commutes with the action of $M_{\xi}$ by $\sigma$. Hence, by Schur's Lemma, $\widetilde{K}_{\sigma,\lambda}(\xi)$ acts by a scalar on each irreducible representation $W_\tau(\xi)$ of $M_\xi$ that occurs in the decomposition \eqref{eq:DecompositionSigmaMxi} of $V_\sigma$. Moreover, by \eqref{align:cnfsn4exfz} this scalar is independent of $\xi$, so we can write
$$ \widetilde{K}_{\sigma,\lambda}(\xi) = \sum_{\tau\preceq\sigma}a_{\sigma,\lambda}(\tau)\cdot\mathrm{id}_{W_\tau(\xi)} \qquad (\xi\in\mathbb{R}^n\setminus\{0\}). $$

To determine the scalars $a_{\sigma,\lambda}(\tau)$, we make use of Proposition~\ref{prop:NactionInFpicture}. Let $f\in\widehat{I}_{\sigma,\lambda}(\tau)$ and $i$ be such that $\tau+e_i$ or $\tau-e_i$ is a highest weight of a representation occurring in $\sigma|_{M_{e_1}}$. Then we have
\begin{align}
\mathrm{pr}_{\widehat{I}_{\sigma,-\lambda}(\tau\pm e_i)}((\widehat{A}_{\sigma,\lambda}\circ d\widehat{\pi}_{\sigma,\lambda}(N_j))f)=\mathrm{pr}_{\widehat{I}_{\sigma,-\lambda}(\tau\pm e_i)}((d\widehat{\pi}_{\sigma,-\lambda}(N_j)\circ\widehat{A}_{\sigma,\lambda})f).\label{eq:IntertwiningRelationProjected}
\end{align}
We first calculate the left hand side using \eqref{eq:NactionInFpicture}:
\begin{align*}
	\mathrm{pr}_{\widehat{I}_{\sigma,-\lambda}(\tau\pm e_i)}((\widehat{A}_{\sigma,\lambda}\circ d\widehat{\pi}_{\sigma,\lambda}(N_j))f) &= a_{\sigma,\lambda}(\tau\pm e_i)\|\xi\|^{-2\lambda}\mathrm{pr}_{\widehat{I}_{\sigma,\lambda}(\tau\pm e_i)}(d\widehat{\pi}_{\sigma,\lambda}(N_j)f)\\
	&=-ia_{\sigma,\lambda}(\tau\pm e_i)\|\xi\|^{-2\lambda}(2\lambda+1\pm 2(\tau_i+\tfrac{n-2i-1}{2}))\mathrm{pr}_{\widehat{I}_{\sigma,-\lambda}(\tau\pm e_i)}(\partial_jf).
\end{align*}
A similar computation can be carried out for the right hand side of \eqref{eq:IntertwiningRelationProjected}, now using \eqref{eq:NactionInFpicture} for $-\lambda$ instead of $\lambda$. Here we have to be slightly more careful, because the intertwiner $\widehat{A}_{\sigma,\lambda}$ multiplies $f$ with $\|\xi\|^{-2\lambda}$ and afterwards $\partial_j$ is applied to the product $\|\xi\|^{-2\lambda}f$. However, since $\partial_j\|\xi\|^{-2\lambda}\cdot f$ is still contained in $\widehat{I}_{\sigma,-\lambda}(\tau)$, its projection to $\widehat{I}_{\sigma,-\lambda}(\tau\pm e_i)$ is zero. This yields the following expression for the right hand side of \eqref{eq:IntertwiningRelationProjected}:
\begin{gather*}
\mathrm{pr}_{\widehat{I}_{\sigma,-\lambda}(\tau\pm e_i)}((d\widehat{\pi}_{\sigma,-\lambda}(N_j)\circ\widehat{A}_{\sigma,\lambda})f) = -ia_{\sigma,\lambda}(\tau)\|\xi\|^{-2\lambda}(-2\lambda+1\pm 2(\tau_i+\tfrac{n-2i-1}{2}))\mathrm{pr}_{\widehat{I}_{\sigma,-\lambda}(\tau\pm e_i)}(\partial_jf).
\end{gather*}
Therefore, we obtain the following recursion for the scalars $a_{\sigma,\lambda}(\tau)$:
\begin{equation}
	(2\lambda+1\pm 2(\tau_i+\tfrac{n-2i-1}{2}))\cdot a_{\sigma,\lambda}(\tau\pm e_i) = (-2\lambda+1\pm 2(\tau_i+\tfrac{n-2i-1}{2}))\cdot a_{\sigma,\lambda}(\tau).\label{eq:RecursionIntertwinerFpicture}
\end{equation}

\begin{theorem}\label{thm:IntertwiningOperatorsFpicture}
	The intertwining operator $\widehat{A}_{\sigma,\lambda}:\widehat{I}_{\sigma,\lambda}\to\widehat{I}_{\sigma,-\lambda}$ is given by
	$$ \widehat{A}_{\sigma,\lambda}f(\xi) = \|\xi\|^{-2\lambda}\sum_{\tau\preceq\sigma}a_{\sigma,\lambda}(\tau)\cdot\mathrm{pr}_{W_\tau(\xi)}f(\xi), $$
	where
	\begin{equation}
		a_{\sigma,\lambda}(\tau) = \gamma(\lambda)\cdot\prod_{i=1}^{m-1}\Bigl(\frac{n}{2}-i+\sigma_{i+1}-\lambda\Bigr)_{|\tau_i|-\sigma_{i+1}}\Bigl(\frac{n}{2}-i+|\tau_i|+\lambda\Bigr)_{\sigma_i-|\tau_i|}\label{eq:ExplicitIntertwiningScalars}
	\end{equation}
	for some entire and nowhere vanishing function $\gamma(\lambda)$ of $\lambda\in\mathbb{C}$.
\end{theorem}

\begin{proof}
	We already argued above that $\widehat{A}_{\sigma,\lambda}$ is of the claimed form for some scalars $a_{\sigma,\lambda}(\tau)$, so it suffices to show the formula for $a_{\sigma,\lambda}(\tau)$. Solving the recursion \eqref{eq:RecursionIntertwinerFpicture} shows that $a_{\sigma,\lambda}(\tau)$ is a multiple of
	\begin{equation}
		\prod_{i=1}^{m-1}\frac{\Gamma(\frac{n}{2}-i+\tau_i-\lambda)}{\Gamma(\frac{n}{2}-i+\tau_i+\lambda)}.\label{eq:EigenvaluesFpictureWithGammaFactors}
	\end{equation}
	This expression is meromorphic in $\lambda$, so by the holomorphicity of $\widehat{A}_{\sigma,\lambda}$ we conclude that the proportionality factor between $a_{\sigma,\lambda}(\tau)$ and \eqref{eq:EigenvaluesFpictureWithGammaFactors} is also meromorphic in $\lambda$. Since $\widehat{A}_{\sigma,\lambda}$ is holomorphic and nowhere vanishing, we just have to find a holomorphic and nowhere vanishing normalization of \eqref{eq:EigenvaluesFpictureWithGammaFactors}. First note that, by the same argument as used in Section~\ref{sec:PrincipalSeriesAndIntertwiningOperators}, we can replace $\tau_i$ by $|\tau_i|$ in \eqref{eq:EigenvaluesFpictureWithGammaFactors}. Multiplying \eqref{eq:EigenvaluesFpictureWithGammaFactors} with the meromorphic function
	$$ \prod_{i=1}^{m-1}\frac{\Gamma(\frac{n}{2}-i+\sigma_i+\lambda)}{\Gamma(\frac{n}{2}-i+\sigma_{i+1}-\lambda)} $$
	produces the desired formula.
\end{proof}

\begin{remark}
	In the previous discussion, we have found an expression for the existing family of intertwining operators. It also seems possible to define intertwining operators by the formulas in Theorem~\ref{thm:IntertwiningOperatorsFpicture}, and this construction would also have the chance of generalizing to intertwining operators defined only on proper subrepresentations of $\widehat{I}_{\sigma,\lambda}$.. However, there are some subtleties that make such constructions complicated. For instance it does not seem obvious that the expression for $\widehat{A}_{\sigma,\lambda}f$ in Theorem~\ref{thm:IntertwiningOperatorsFpicture} is actually contained in $\widehat{I}_{\sigma,-\lambda}$ for every $f\in\widehat{I}_{\sigma,\lambda}$.
\end{remark}

\begin{remark}\label{rem:SpecialCasePforms}
For $V_\sigma=\bigwedge^p\mathbb{C}^n$ the formula \eqref{eq:ExplicitIntertwiningScalars} was previously obtained by Fischmann--{\O}rsted~\cite[Corollary 4.2 and Remark 4.10]{FischmannOrsted2021}. In this case $\sigma=e_1+\cdots+e_p=(1,\ldots,1,0,\ldots,0)$, so the possible values of $\tau$ are $e_1+\cdots+e_p$ and $e_1+\cdots+e_{p-1}$:
$$ a_{\sigma,\lambda}(e_1+\cdots+e_p) = \gamma(\lambda)\cdot\left(\frac{n}{2}-p-\lambda\right) \qquad \mbox{and} \qquad a_{\sigma,\lambda}(e_1+\cdots+e_{p-1}) = \gamma(\lambda)\cdot\left(\frac{n}{2}-p+\lambda\right). $$
\end{remark}

\subsection{Explicit Hilbert spaces in the F-picture}

We use the expression for the intertwining operators $\widehat{A}_{\sigma,\lambda}$ obtained in Theorem~\ref{thm:IntertwiningOperatorsFpicture} to construct explicit Hilbert spaces of vector-valued $L^2$-functions for the unitarizable composition factors.

\begin{theorem}\label{thm:InnerProductsInFpicture}
  Let $\sigma\in\widehat{M}$, $a=a_\sigma=\min\{i:\sigma_{i+1}=0\}$. Then 
  \begin{equation*}
    (f_1,f_2) \mapsto \sum_{\tau\preceq\sigma}a(\tau)\int_{\mathbb{R}^n\setminus\{0\}} \langle\mathrm{pr}_{W_\tau(\xi)}f_1(\xi), \mathrm{pr}_{W_\tau(\xi)}f_2(\xi) \rangle_{\sigma } \lVert \xi \rVert^{-2\operatorname{Re}\lambda} d \xi
  \end{equation*}
  defines a $G$-invariant inner product with respect to $\widehat{\pi}_{\sigma,\lambda}$ on
  \begin{enumerate}
	\item\label{thm:InnerProductsInFpicture1} the unitary principal series $\widehat{I}_{\sigma,\lambda}$ ($\lambda\in i\mathbb{R}$) if $a(\tau)=1$ for all $\tau\preceq\sigma$.
	\item\label{thm:InnerProductsInFpicture2} the complementary series $\widehat{I}_{\sigma,\lambda}$ ($\sigma_m=0$ and $|\lambda|<\rho-a$) if $a(\tau)=a_{\sigma,\lambda}(\tau)$ for all $\tau\preceq\sigma$.
	\item\label{thm:InnerProductsInFpicture3} the subrepresentation $\mathfrak{I}(\sigma , a, j)$ (for $0\leq a<\frac{n-1}{2}$) or $\mathfrak{I}(\sigma , \frac{n-1}{2},j)^{+}$ and $\mathfrak{I}(\sigma , \frac{n-1}{2},j)^{-}$ (for $a=\frac{n-1}{2}$) of $\widehat{I}_{\sigma,\lambda}$ (with $0\leq j<\sigma_a$ and $\lambda=\rho-a+j$) if
    \begin{equation*}
      a(\tau) = \frac{1}{a_{\sigma,-\lambda}(\tau)} = \prod_{i=1}^a\frac{1}{(n-a-i+\sigma_{i+1}+j)_{|\tau_i|-\sigma_{i+1}}(a-i+|\tau_i|-j)_{\sigma_i-|\tau_i|}} \qquad (|\tau_a|>j),
    \end{equation*}
	\item\label{thm:InnerProductsInFpicture4} the quotient $\mathfrak{Q}(\sigma , a, 0)$ (for $0<a<m$ and $j=0$) of $\widehat{I}_{\sigma,\lambda}$ (with $\lambda=\rho-a$) if
    \begin{equation*}
      a(\tau) = a_{\sigma,\lambda}(\tau) = \prod_{i=1}^a\Big(a-i+\sigma_{i+1}\Big)_{\tau_i-|\sigma_{i+1}|}\Big(n-a-i+\tau_i\Big)_{\sigma_i-|\tau_i|}. 
    \end{equation*}
  \end{enumerate}
\end{theorem}

\begin{proof}
  We recall the classical result that the sesquilinear pairing
  \begin{equation*}
    I_{\sigma , \lambda} \times I_{\sigma , -\overline{\lambda}} \to \mathbb{C}, \quad (\varphi , \psi) \mapsto \int_{\mathbb{R}^n} \langle \varphi (x), \psi(x) \rangle_{\sigma } \mathrm{d} x
  \end{equation*}
  is $G$-invariant. Since $\mathcal{F}$ is an isometry, this pairing gives rise to a $G$-invariant pairing in the F-picture given by the same formula. Writing $f_i(\xi)=\sum_{\tau\preceq\sigma}\mathrm{pr}_{W_\tau(\xi)}f_i(\xi)$ shows the claim for the unitary principal series. Twisting with the intertwining operator $\widehat{A}_{\sigma,\lambda}$, we obtain for every $\lambda\in\mathbb{R}$ a $G$-invariant Hermitian form
  \begin{equation*}
    \widehat{I}_{\sigma , \lambda} \times \widehat{I}_{\sigma,\lambda} \rightarrow \mathbb{C}, \quad(f_1 ,f_2) \mapsto \int_{\mathbb{R}^n\setminus\{0\}} \langle f_1 (\xi), \widehat{A}_{\sigma , \lambda }f_2(\xi) \rangle_{\sigma } \mathrm{d} \xi.
  \end{equation*}
  Writing $f_i(\xi)=\sum_{\tau\preceq\sigma}\mathrm{pr}_{W_\tau(\xi)}f_i(\xi)$ and applying Theorem~\ref{thm:IntertwiningOperatorsFpicture}, this expression can be written as
  $$ \sum_{\tau\preceq\sigma}a_{\sigma,\lambda}(\tau)\int_{\mathbb{R}^n\setminus\{0\}} \langle\mathrm{pr}_{W_\tau(\xi)}f_1(\xi), \mathrm{pr}_{W_\tau(\xi)}f_2(\xi) \rangle_{\sigma } \lVert \xi \rVert^{-2\lambda} d \xi. $$
  This shows the claim for the complementary series. For $\lambda=\rho-a$, this induces a Hermitian form on the quotient $\mathfrak{Q}(\sigma,a,0)=\widehat{I}_{\sigma,\lambda}/\ker\widehat{A}_{\sigma,\lambda}$, so the claimed formula for this case follows. This form is in fact positive semidefinite on $\widehat{I}_{\sigma,\lambda}$ and positive definite on the quotient since $a_{\sigma,\lambda}(\tau)\geq0$ for all $\tau\preceq\sigma$ and $a_{\sigma,\lambda}(\tau)=0$ if and only if $\tau_a>0$. To show the claimed formulas for the subrepresentations $\mathfrak{I}(\sigma,a,j)$ and $\mathfrak{I}(\sigma,\frac{n-1}{2},j)^+\oplus\mathfrak{I}(\sigma,\frac{n-1}{2},j)^-$, we view them as the image of $\widehat{A}_{\sigma,-\lambda}$ and consider the invariant Hermitian form
  $$ \operatorname{im}(\widehat{A}_{\sigma,-\lambda})\times\operatorname{im}(\widehat{A}_{\sigma,-\lambda})\to\mathbb{C}, \quad (f_1,f_2)=(\widehat{A}_{\sigma,-\lambda}g_1,\widehat{A}_{\sigma,-\lambda}g_2)\mapsto\int_{\mathbb{R}^n\setminus\{0\}}\langle g_1(\xi),\widehat{A}_{\sigma,-\lambda}g_2(\xi)\rangle_\sigma\,d\xi. $$
  Again, by Theorem~\ref{thm:IntertwiningOperatorsFpicture} this expression equals
  $$ \sum_{\tau\preceq\sigma}\frac{1}{a_{\sigma,-\lambda}(\tau)}\int_{\mathbb{R}^n\setminus\{0\}} \langle\mathrm{pr}_{W_\tau(\xi)}f_1(\xi), \mathrm{pr}_{W_\tau(\xi)}f_2(\xi) \rangle_{\sigma } \lVert \xi \rVert^{-2\lambda} d \xi. $$
  This implies the remaining formulas.
\end{proof}

\begin{remark}
	Since invariant Hermitian forms on irreducible representations are unique up to scalar multiples, the above arguments can also be used to decide whether a subrepresentation or quotient of $\widehat{I}_{\sigma,\lambda}$ is unitarizable or not by inspecting the positivity of the relevant coefficients $a_{\sigma,\lambda}(\tau)$.
\end{remark}

\section{Applications}

We present two applications of the new models for the irreducible admissible representations of $G$. The first one is a simple proof of the unitary branching laws for the restriction to the parabolic subgroup $\overline{P}$ obtained by Liu--Oshima--Yu~\cite{LiuOshimaYu23}, and the second one is a description of the space of Whittaker vectors.

\subsection{Branching laws with respect to \texorpdfstring{$\overline{P}$}{conj(P)}}

In \cite{LiuOshimaYu23} the authors obtain explicit branching laws for all irreducible unitary representations of $G$ when restricted to $P$. Their proof relies on first calculating a specific functor on the smooth part of principal series representations and then using properties of this functor to derive the branching laws. We now describe how to derive these branching laws from our results.

The infinite-dimensional irreducible unitary representations of $\overline{P}$ can be obtained as induced representations. We fix the unitary character $\psi$ of $\overline{N}$ given by
\begin{equation}
	\psi(\overline{n}_x) = e^{-i\langle x,e_1\rangle} \qquad (x\in\mathbb{R}^n)\label{eq:DefinitionPsi}
\end{equation}
and note that its stabilizer in $MA$ is equal to $M_{e_1}$. For every $\tau\in\widehat{M}_{e_1}$ the induced representation
$$ L^2\mathrm{Ind}_{M_{e_1}\overline{N}}^{\overline{P}}(\tau\otimes\psi) = \left\{f:\overline{P}\to W_\tau:\begin{array}{l}f(\overline{p}m\overline{n})=\psi(\overline{n})^{-1}\tau(m)^{-1}f(\overline{p})\mbox{ for all }m\in M_{e_1},\overline{n}\in\overline{N},\\\mbox{$f$ is measurable and }\int_{M/M_{e_1}\times A}\|f(ma)\|_\tau^2\,d(ma)<\infty\end{array}\right\} $$
with the left regular action of $\overline{P}$ is irreducible and unitary. Every infinite-dimensional irreducible unitary representation of $\overline{P}$ is of this form for precisely one $\tau\in\widehat{M}_{e_1}$. In what follows we also need a smooth version of the induced representation, so we consider
$$ C^\infty\mathrm{Ind}_{M_{e_1}\overline{N}}^{\overline{P}}(\tau\otimes \psi) = \left\{f:\overline{P}\to W_\tau\mbox{ smooth}:f(\overline{p}m\overline{n})=\psi(\overline{n})^{-1}\tau(m)^{-1}f(\overline{p})\mbox{ for all }m\in M_{e_1},\overline{n}\in\overline{N}\right\} $$
and the subspace $C_c^\infty\mathrm{Ind}_{M_{e_1}\overline{N}}^{\overline{P}}(\tau\otimes \psi)$ of functions which are compactly supported modulo $M_{e_1}\overline{N}$.

The following lemma identifies the subspaces $\widehat{I}_{\sigma,\lambda}(\tau)$ with induced representations:

\begin{lemma}[cf.\@ {\cite[Lemma 3.10]{LiuOshimaYu23}}]\label{lem:PbarReduction}
	Let $\sigma\in\widehat{M}$ and $\lambda\in\mathbb{C}$. For every $\tau \in \widehat{M}_{e_1}$ the map
	$$ \widehat{I}_{\sigma,\lambda}(\tau) \to C^\infty\mathrm{Ind}_{M_{e_1}\overline{N}}^{\overline{P}}(\tau\otimes \psi), \quad f\mapsto f_\tau, \quad f_\tau(\overline{p}) = (\widehat{\pi}_{\sigma,\lambda}(\overline{p})^{-1}f)(e_1) \qquad (\overline{p}\in\overline{P}) $$
	is a $\overline{P}$-equivariant embedding whose image contains $C_c^\infty\mathrm{Ind}_{M_{e_1}\overline{N}}^{\overline{P}}(\tau\otimes \psi)$. Moreover,
	\begin{equation}
		\|f_\tau\|^2_{L^2} = \int_{\mathbb{R}^n\setminus\{0\}}\|f(\xi)\|_{W_\tau(\xi)}^2\|\xi\|^{-2\lambda}\,d\xi \qquad \mbox{for all }f\in\widehat{I}_{\sigma,\lambda}(\tau).\label{eq:IsometryFormulaPbarReduction}
	\end{equation}
\end{lemma}

In the last identity we also allow both sides to be infinite.

\begin{proof}
	By \cite[Lemma 3.9]{LiuOshimaYu23} every $f\in\widehat{I}_{\sigma,\lambda}(\tau)$ is smooth on $\mathbb{R}^n\setminus\{0\}$, so its evaluation at $e_1$ is well-defined and by the definition of $\widehat{I}_{\sigma,\lambda}(\tau)$ contained in $W_\tau$. For any $m \in M_{e_1},\ \overline{n}_{x} \in \overline{N}$, and $\overline{p} \in \overline{P}$ we have
	\begin{equation*}
		f_\tau(\overline{p}m \overline{n}_x) =(\widehat{\pi}_{\sigma , \lambda}(\overline{n}_{x}^{-1}m^{-1} \overline{p}^{-1})f)(e_1)=e^{i \langle x, e_1  \rangle} \sigma(m^{-1} )(\widehat{\pi}_{\sigma, \lambda }(\overline{p}^{-1})f)(m e_1)=(\tau \otimes e^{-i e_1^*})(m, \overline{n}_x)^{-1} f_\tau(\overline{p}),
	\end{equation*}
	where we used that $m e_1 = e_1$ and that $\widehat{I}_{\sigma,\lambda}(\tau)$ is $\overline{P}$-invariant. Thus, $f_\tau\in C^\infty \mathrm{Ind}_{M_{e_1}\overline{N}}^{\overline{P}}(\tau \otimes e^{-i e_1^*})$, and the $\overline{P}$-equivariance is immediate from the definition. 
	That the map $f\mapsto f_\tau$ is injective follows from the fact that $MA$ acts transitively on $\mathbb{R}^n\setminus\{0\}$. Finally, the claimed identity for the $L^2$-norms is a consequence of the action of $MA$ in $\widehat{\pi}_{\sigma,\lambda}$ (see \eqref{eq:ActionFpictureM} and \eqref{eq:ActionFpictureA}) and the polar coordinates formula
	\begin{equation*}
		\int_{M/M_{e_1}}\int_{\mathbb{R}}\varphi(e^t\cdot me_1)\,dt\,dm = \int_{\mathbb{R}^n\setminus\{0\}}\varphi(\xi)\frac{d\xi}{\|\xi\|^n}, \qquad \varphi\in L^1_{\mathrm{loc}}(\mathbb{R}^n\setminus\{0\}).\qedhere
	\end{equation*}
\end{proof}

\begin{corollary}\label{cor:BranchingPbar}
	Let $\sigma\in\widehat{M}$ and $a=a_\sigma=\min\{i:\sigma_{i+1}=0\}$. The following branching laws hold:
	\begin{itemize}
		\item (unitary principal series and complementary series) For $\lambda\in i\mathbb{R}$ or $\sigma_m=0$ and $|\lambda|<\rho-a$:
		$$ \widehat{I}_{\sigma,\lambda}|_{\overline{P}} \simeq \bigoplus_{\tau\preceq\sigma}L^2\mathrm{Ind}_{M_{e_1}\overline{N}}^{\overline{P}}(\tau\otimes \psi). $$
		\item (unitarizable subrepresentations) For $0\leq a<m$ and $0\leq j<\sigma_a$:
		\begin{align*}
			\mathfrak{I}(\sigma , a, j)|_{\overline{P}} &\simeq \bigoplus_{\substack{\tau\preceq\sigma\\\tau_a>j}}L^2\mathrm{Ind}_{M_{e_1}\overline{N}}^{\overline{P}}(\tau\otimes \psi) && \mbox{if }a\neq\tfrac{n-1}{2},\\
			\mathfrak{I}(\sigma , a, j)^\pm|_{\overline{P}} &\simeq \bigoplus_{\substack{\tau\preceq\sigma\\\pm\tau_a>j}}L^2\mathrm{Ind}_{M_{e_1}\overline{N}}^{\overline{P}}(\tau\otimes \psi) && \mbox{if }a=\tfrac{n-1}{2}.
		\end{align*}
		\item (unitarizable quotients) For $0<a<m$:
		\begin{equation*}
			\mathfrak{Q}(\sigma , a, 0)|_{\overline{P}} \simeq \bigoplus_{\substack{\tau\preceq\sigma\\\tau_a=0}}L^2\mathrm{Ind}_{M_{e_1}\overline{N}}^{\overline{P}}(\tau\otimes \psi)
		\end{equation*}
	\end{itemize}
\end{corollary}

\begin{proof}
	By Proposition~\ref{prop:SubrepsAndQuotientsInFpicture} and Theorem~\ref{thm:InnerProductsInFpicture}, each of the Hilbert spaces on which the representations are realized is a direct sum of the completions of the relevant subspaces $\widehat{I}_{\sigma,\lambda}(\tau)$ with respect to the $L^2$-inner product given by the right hand side of \eqref{eq:IsometryFormulaPbarReduction}. The claim now follows from Lemma~\ref{lem:PbarReduction}.
\end{proof}

\begin{remark}
	To be able to compare our results with the ones in \cite{LiuOshimaYu23}, we match their notation with ours. The representations constructed in \cite{LiuOshimaYu23} are denoted by $\pi_j(\gamma)$, $\pi_j'(\gamma)$ and $\pi^\pm(\gamma)$, where $\gamma$ denotes the infinitesimal character. We identify each of these representations with one of our representations $\mathfrak{I}(\sigma,i,j)$, $\mathfrak{I}(\sigma,i,j)^\pm$ and $\mathfrak{Q}(\sigma,i,j)$. Then it is easy to see that Corollary~\ref{cor:BranchingPbar} is the same as \cite[Theorems 3.20--3.24]{LiuOshimaYu23}.
	\begin{enumerate}[(i)]
		\item If $n$ is odd:
		\begin{itemize}
			\item For $a \in \left\{0, \ldots , \frac{n-3}{2}\right\}$ we have $\mathfrak{I}(\sigma , a, j) \cong \pi_a '(\gamma)$ and $\mathfrak{Q}(\sigma ,a , j) \cong \pi_a(\gamma)$, where
			$$ \gamma = \left(\sigma_1 + \rho -1, \ldots , \sigma_a + \rho -a , \sigma_{a+1}+\rho -a +j, \sigma_{a+1}+ \rho -a-1, \sigma_{a+2}+\rho -a-2, \ldots, \sigma_{\frac{n-1}{2}}+ \frac{1}{2}\right). $$
			\item For $a = \frac{n-1}{2}$ we have $\mathfrak{I}\left(\sigma , \frac{n-1}{2}, j\right)^{\pm} \cong \pi^{\pm}(\gamma)$ and $\mathfrak{Q}\left(\sigma , \frac{n-1}{2}, j\right) \cong \pi_{\frac{n-1}{2}}(\gamma)$, where
			$$ \gamma =\left(\sigma_1 + \rho -1, \ldots , \sigma_{\frac{n-1}{2}}+\frac{1}{2}, j + \frac{1}{2}\right). $$
		\end{itemize}
		Note that in this case, the representations $\pi^{\pm}(\gamma)$ are discrete series representations with lowest $K$-type $(\sigma_1 , \ldots , \sigma_{n-1}, \pm(j+1))$.
		\item If $n$ is even:
		For $a \in \left\{1, \ldots , \frac{n}{2}-1\right\}$ we have $\mathfrak{I}(\sigma , a, j) \cong \pi_a '(\gamma)$ and $\mathfrak{Q}(\sigma ,a , j) \cong \pi_a(\gamma)$, where
		$$ \gamma = \left(\sigma_1 + \rho -1, \ldots , \sigma_a + \rho -a , \sigma_{a+1}+\rho -a +j, \sigma_{a+1}+ \rho -a-1, \sigma_{a+2}+\rho -a-2, \ldots, \sigma_{\frac{n}{2}}\right). $$
	\end{enumerate}
\end{remark}

\subsection{Whittaker vectors}

As in the previous section, we fix the unitary character $\psi$ of $\overline{N}$ given by \eqref{eq:DefinitionPsi}. For a Casselman--Wallach representation $(\pi,V)$ of $G$ we write
$$ \mathrm{Wh}_\psi(\pi) = \{W\in V':W(\pi(\overline{n})v) = \psi(\overline{n})W(v)\mbox{ for all }v\in V,\overline{n}\in\overline{N}\} $$
for the space of Whittaker vectors on $\pi$ with respect to $\psi$. Since $M_{e_1}$ is the stabilizer of $\psi$, the space $\mathrm{Wh}_\psi(\pi)$ carries an action of $M_{e_1}$ by
$$ m\cdot W \coloneqq W\circ\pi(m)^{-1} \qquad (m\in M_{e_1},W\in\mathrm{Wh}_\psi(\pi)). $$

Using Proposition~\ref{prop:SubrepsAndQuotientsInFpicture} we can easily describe $\mathrm{Wh}_\psi(\pi)$ for all irreducible Casselman--Wallach representations of $G$. To do so, we first describe $\mathrm{Wh}_\psi(\pi_{\sigma,\lambda})$ for all $\sigma\in\widehat{M}$ and $\lambda\in\mathbb{C}$.

\begin{lemma}
	For every $\sigma\in\widehat{M}$ and $\lambda\in\mathbb{C}$ we have
	$$ \mathrm{Wh}_\psi(\pi_{\sigma,\lambda}) = \{W_{\sigma,\lambda}^\eta:\eta\in V_\sigma'\}, $$
	where $W_{\sigma,\lambda}^\eta\in I_{\sigma,\lambda}'$ is the unique functional satisfying
        $$ W_{\sigma,\lambda}^\eta(\varphi) = \int_{\mathbb{R}^n} \eta(\varphi (x))\psi(\overline{n}_x)\,dx \qquad \mbox{for all }\varphi \in\mathcal{S}(\mathbb{R}^n)\otimes V_\sigma\subseteq I_{\sigma,\lambda}. $$
      \end{lemma}

      \begin{proof}
	By the definition of the Fourier transform we have
	$$ W_{\sigma,\lambda}^\eta(\varphi) = (2\pi)^{\frac{n}{2}}\cdot\eta(\widehat{\varphi}(e_1)) \qquad \mbox{for all }\varphi\in\mathcal{S}(\mathbb{R}^n)\otimes V_\sigma. $$
	The right hand side makes sense for all $\widehat{\varphi}=f\in\widehat{I}_{\sigma,\lambda}$ since every such $f$ is smooth on $\mathbb{R}^n\setminus\{0\}$, so we can take this as a definition for $W_{\sigma,\lambda}^\eta$. In the F-picture $\widehat{\pi}_{\sigma,\lambda}$, the equivariance condition for a functional $\widehat{W}\in\widehat{I}_{\sigma,\lambda}'$ to be a Whittaker vector with respect to $\psi$ reads (see \eqref{eq:ActionFpictureNbar})
	$$ \widehat{W}(e^{i\langle x,\cdot\rangle}f) = e^{i\langle x,e_1\rangle}\widehat{W}(f) \qquad \mbox{for all }f\in\widehat{I}_{\sigma,\lambda}. $$
	Since $C_c^\infty(\mathbb{R}^n\setminus\{0\})\otimes V_\sigma$ is contained in $\widehat{I}_{\sigma,\lambda}$, we can restrict $\widehat{W}$ to $C_c^\infty(\mathbb{R}^n\setminus\{0\})\otimes V_\sigma$ and hence view it as a distribution $\widehat{W}\in\mathcal{D}'(\mathbb{R}^n\setminus\{0\})\otimes V_\sigma'$. The equivariance condition implies that this distribution is supported on $\{e_1\}$ and of order $0$, hence $\widehat{W}(f)=\eta(f(e_1))$ for some $\eta\in V_\sigma'$. Conversely, every such distribution satisfies the above equivariance condition and hence gives rise to a Whittaker vector.
\end{proof}

The previous statement classifies Whittaker vectors for all irreducible principal series representations $\pi_{\sigma,\lambda}$. Moreover, since
$$ m\cdot W_{\sigma,\lambda}^\eta = W_{\sigma,\lambda}^{\sigma'(m)\eta} \qquad (m\in M_{e_1},\eta\in V_\sigma'), $$
where $(\sigma',V_\sigma')$ denotes the contragredient representation of $(\sigma,V_\sigma)$, the $M_{e_1}$-module structure of $\mathrm{Wh}_\psi(\pi_{\sigma,\lambda})$ is given in terms of the decomposition of $V_\sigma'$ into $W_\tau'$ ($\tau\preceq\sigma$) dual to \eqref{eq:DecompositionSigmaIntoTau}:
$$ \mathrm{Wh}_\psi(\pi_{\sigma,\lambda}) = \bigoplus_{\tau\preceq\sigma}\{W_{\sigma,\lambda}^\eta:\eta\in W_\tau'\} \simeq \bigoplus_{\tau\preceq\sigma}W_\tau'. $$

To determine the space of Whittaker vectors on quotients of $\pi_{\sigma,\lambda}$, it is clearly sufficient to determine which Whittaker vectors on $\pi_{\sigma,\lambda}$ vanish on the subrepresentation that is quotiented out. It turns out that Whittaker vectors on subrepresentations are all given by restrictions of $W_{\sigma,\lambda}^\eta$. Together, we obtain the following description of the space of Whittaker vectors for all irreducible Casselman--Wallach representations of $G$:

\begin{corollary}\label{cor:WhittakerVectors}
	Let $\sigma\in\widehat{M}$ and $a=a_\sigma=\min\{i:\sigma_{i+1}=0\}$.
	\begin{itemize}
		\item If $\pi_{\sigma,\lambda}$ is irreducible, then
		$$ \mathrm{Wh}_\psi(\pi_{\sigma,\lambda}) = \bigoplus_{\tau\preceq\sigma}\{W_{\sigma,\lambda}^\eta:\eta\in W_\tau'\} \simeq \bigoplus_{\tau\preceq\sigma}W_\tau'. $$
		\item For $0\leq i\leq a$ and $0\leq j<\sigma_i-|\sigma_{i+1}|$ we have
		\begin{align*}
			\mathrm{Wh}_\psi(\mathfrak{I}(\sigma,i,j)) &= \bigoplus_{\substack{\tau\preceq\sigma\\\tau_i>|\sigma_{i+1}|+j}} \{W_{\sigma,\lambda}^\eta|_{\mathfrak{I}(\sigma,i,j)}:\eta\in W_\tau'\} \simeq \bigoplus_{\substack{\tau\preceq\sigma\\\tau_i>|\sigma_{i+1}|+j}}W_\tau' && \mbox{if }i\neq\tfrac{n-1}{2},\\
			\mathrm{Wh}_\psi(\mathfrak{I}(\sigma,i,j)^\pm) &= \bigoplus_{\substack{\tau\preceq\sigma\\\pm\tau_i>|\sigma_{i+1}|+j}} \{W_{\sigma,\lambda}^\eta|_{\mathfrak{I}(\sigma,i,j)}:\eta\in W_\tau'\} \simeq \bigoplus_{\substack{\tau\preceq\sigma\\\pm\tau_i>|\sigma_{i+1}|+j}}W_\tau' && \mbox{if }i=\tfrac{n-1}{2},
		\end{align*}
		and
		$$ \mathrm{Wh}_\psi(\mathfrak{Q}(\sigma,i,j)) = \bigoplus_{\substack{\tau\preceq\sigma\\\tau_i\leq|\sigma_{i+1}|+j}} \{W_{\sigma,\lambda}^\eta:\eta\in W_\tau'\} \simeq \bigoplus_{\substack{\tau\preceq\sigma\\\tau_i\leq|\sigma_{i+1}|+j}}W_\tau', $$
		where we put $\tau_0=+\infty$.
	\end{itemize}
\end{corollary}

\begin{proof}
  Again, we solve the problem in the F-picture $\widehat{I}_{\sigma,\lambda}$ instead of the non-compact picture $I_{\sigma,\lambda}$. Since the decomposition \eqref{eq:DecompositionIhat} is $\overline{N}$-invariant, it suffices to study $\overline{N}$-equivariant functionals on each $\widehat{I}_{\sigma,\lambda}(\tau)$. But for $f\in\widehat{I}_{\sigma,\lambda}(\tau)$ we have $f(e_1)\in W_\tau$, so $\eta(f(e_1))$ is zero unless $\eta\in W_\tau'$. Together with the description of the composition factors in terms of the spaces $\widehat{I}_{\sigma,\lambda}(\tau)$ in Proposition~\ref{prop:SubrepsAndQuotientsInFpicture} this shows the claim.
\end{proof}

\providecommand{\bysame}{\leavevmode\hbox to3em{\hrulefill}\thinspace}
\providecommand{\MR}{\relax\ifhmode\unskip\space\fi MR }
\providecommand{\MRhref}[2]{%
  \href{http://www.ams.org/mathscinet-getitem?mr=#1}{#2}
}
\providecommand{\href}[2]{#2}


\end{document}